\documentclass{article}

\usepackage[english]{babel}

\usepackage[letterpaper,top=2cm,bottom=2cm,left=2cm,right=2cm,marginparwidth=1.75cm]{geometry}

\usepackage{calrsfs}
\usepackage{amssymb}
\usepackage{amsmath}
\usepackage{amsthm}
\usepackage{subcaption}
\usepackage{multirow}
\usepackage{graphicx}
\usepackage{amsfonts}
\usepackage{mathtools}
\usepackage{mathrsfs}

\newtheorem{definition}{Definition}
\newtheorem{assumption}{Assumption}
\newtheorem{theorem}{Theorem}

\usepackage[mathscr]{euscript}
\usepackage[dvipsnames]{xcolor}
\usepackage[colorlinks=true, allcolors=OliveGreen]{hyperref}
\usepackage{filecontents}
\usepackage{tikz}
\usepackage{pgfplots}
\usetikzlibrary{external}
\usepackage{authblk}
\usepackage{todonotes}
\usepackage{array}
\newcolumntype{C}[1]{>{\centering\arraybackslash}p{#1}}

\newcommand{\averagel}{\{\!\!\{}
\newcommand{\averager}{\}\!\!\}}

\newcommand{\jumpl}{[\![}
\newcommand{\jumpr}{]\!]}

\newcommand{\jjumpl}{[\![\![}
\newcommand{\jjumpr}{]\!]\!]}

\newcommand{\partition}{\mathcal{T}_h}
\newcommand{\facesinternal}{\mathcal{F}^\mathrm{I}_h}
\newcommand{\faces}{\mathcal{F}_h}
\newcommand{\facesN}{\mathcal{F}_h^N}
\newcommand{\facesNj}{\mathcal{F}_h^{N_j}}
\newcommand{\facesD}{\mathcal{F}_h^D}
\newcommand{\facesDj}{\mathcal{F}_h^{D_j}}
\newcommand{\facesboundary}{\mathcal{F}^\mathrm{B}_h}

\newcommand{\Vh}{\mathrm{\mathbf{V}}_h^\mathrm{DG}}
\newcommand{\Qh}{Q_h^\mathrm{DG}}

\DeclareMathAlphabet{\mathcalligra}{T1}{calligra}{m}{n}

\newtheorem{proposition}{Proposition}

\title{Numerical Modelling of the Brain Poromechanics by High-Order Discontinuous Galerkin Methods\footnote{\textbf{Funding}: PFA has been partially funded by the research grants PRIN2017 n. 201744KLJL. PFA, LD and AMQ has been partially funded by the research grants PRIN2020 n. 20204LN5N5 funded by the Italian Ministry of Universities and Research (MUR). MC, PFA, LD and AMQ are members of INdAM-GNCS.}}

\author[1]{Mattia Corti}

\affil[1]{MOX-Dipartimento di Matematica, Politecnico di Milano, Piazza Leonardo da Vinci 32, Milan, 20133, Italy}
            
\author[1]{Paola F. Antonietti}
\author[1]{Luca Dede'}
\author[1,2]{Alfio Maria Quarteroni}

\affil[2]{Institute of Mathematics, \'{E}cole Polytechnique F\'{e}d\'{e}rale de Lausanne, Station 8, Av. Piccard, Lausanne, CH-1015, Switzerland (Professor Emeritus)}

\begin{document}
\maketitle

\begin{abstract}
We introduce and analyze a discontinuous Galerkin method for the numerical modelling of the equations of Multiple-Network Poroelastic Theory (MPET) in the dynamic formulation. The MPET model can comprehensively describe functional changes in the brain considering multiple scales of fluids. Concerning the spatial discretization, we employ a high-order discontinuous Galerkin method on polygonal and polyhedral grids and we derive stability and a priori error estimates. The temporal discretization is based on a coupling between a Newmark $\beta$-method for the momentum equation and a $\theta$-method for the pressure equations. After the presentation of some verification numerical tests, we perform a convergence analysis using an agglomerated mesh of a geometry of a brain slice. Finally we present a simulation in a three dimensional patient-specific brain reconstructed from magnetic resonance images. The model presented in this paper can be regarded as a preliminary attempt to model the perfusion in the brain.
\end{abstract}
\section{Introduction}
Poroelasticity models the interaction among fluid flow and elastic deformations in porous media. The precursor Biot's equations \cite{biot41} are able to correctly model the physical problems; however, complete and detailed modelling sometimes requires a splitting of the fluid component into multiple distinct network fields \cite{zienkiewiczMPET1982}. Despite initially the multiple networks poroelastic (MPET) model being applied to soil mechanics \cite{zienkiewiczDynamicBehaviourSaturated1984}, more recently, the separation of fluid networks was proposed in the context of biological flows. Indeed, to model blood perfusion, it is essential to separate the vascular network into its fundamental components (arteries, capillaries and veins). This is relevant both in the heart \cite{BarnafiHeartPerfusion22,barnafi_mathematical_2021}, and the brain \cite{tullyCerebralWaterTransport2011,vardakis_exploring_2020} modelling.
\par
In the context of neurophysiology, where blood constantly perfuses the brain and provides oxygen to neurons, the multiple network porous media models have been used to study circulatory diseases, such as ischaemic stroke \cite{jozsa_porous_2021,jozsa_sensitivity_2021}. The cerebrospinal fluid (CSF) that surrounds the brain parenchyma is related to disorders of the central nervous system (CNS), such as hydrocephalus \cite{chouFullyDynamicMulticompartmental2016,vardakisInvestigatingCerebralOedema2016}, and plays a role in CNS clearance, particularly important in Alzheimer's disease, which is strongly linked to the accumulation of misfolded proteins, such as amyloid beta $(\mathrm{A}\beta)$ \cite{guoMultiplenetworkPoroelasticModel2020,thompson_protein-protein_2020,weickenmeier_physics-based_2019,brennan_role_2022}.
\par
Despite the MPET equations find application in different physical contexts, at the best of our knowledge, a complete analysis of the numerical discretization in the dynamic case is still missing. Concerning the discretization of the quasi-static MPET equations, some works proposed an analysis using both the Mixed Finite Element Method \cite{lee_mixed_2019,piersanti_parameter_2021} and the Hybrid High-Order Method (HHO) \cite{botti_hybrid_2021}. The quasi-static version neglects the second-order derivative of the displacement in the momentum balance equation. The physical meaning of neglecting this term is that inertial forces have small impact on the evolution of the fields. However, this term ought to be considered in the application to brain physiology, because of the strong impact of systolic pressure variations on the vascular and tissue deformation \cite{tullyCerebralWaterTransport2011}. In the context of applications, the fully-dynamic system has been applied to model the aqueductal stenosis effects \cite{chouFullyDynamicMulticompartmental2016}.
\par
From a numerical perspective, the discretization of second-order time-dependent problems is challenging. In this work, the time discretization scheme applies a Newmark-$\beta$ method \cite{Newmarkbetamethod} for the momentum equation. Due to the system structure, a continuity equation for each pressure field requires a temporal discretization method for first-order ODEs. We choose the application of a $\theta-$method in this work.
\par
In terms of accuracy, to guarantee low numerical dispersion and dissipation errors, high-order discretization methods are required, cf. for example \cite{quarteroni:EDP}. In this work for space discretization, we proposed a high-order Discontinuous Galerkin formulation on polygonal/polyhedral grids (PolyDG). The PolyDG methods are naturally oriented to high-order approximations. Another strength of the proposed formulation is its flexibility in mesh generation; due to the applicability to polygonal/polyhedral meshes. Indeed, the geometrical complexity of the brain is one of the challenges that need to be considered. The possibility of refining the mesh only in some regions, handling the hanging nodes and eventually using elements which are not tetrahedral, is easy to implement in our approach. For all these reason, intense research has been undertaken on this topic \cite{antonietti_review_2016,houston:book,houston:paper,cangianiVersionDiscontinuousGalerkin2022,bassi_flexibility_2012}, in particular concerning porous media and elasticity in the context of geophysical applications \cite{antonietti_highorder_2021,antonietti_high-order_2020,antonietti_high-order_2022}. Moreover, PolyDG methods exhibit low numerical dispersion and dissipation errors, as recently shown in \cite{antoniettiHighorderDiscontinuousGalerkin2018} for the elastodynamics equations.
\par
\bigskip
The paper is organized as follows. Section \ref{sec:model} introduces the mathematical model of MPET, proposing also some changes for the adaptation to the brain physiology. In Section \ref{sec:polydg}, we introduce the PolyDG space discretization of the problem. In Section \ref{sec:stability} we prove stability of the semi-discretized MPET system in a suitable (mesh dependent) version. Section \ref{sec:error} is devoted to the proof of a priori error estimates of the semi-discretized MPET problem. In Section \ref{sec:temporaldisc}, we introduce a temporal discretization by means of Newmark-$\beta$ and $\theta$-methods. In Section \ref{sec:numericalresults} we show some numerical results considering convergence tests with analytical solutions. Moreover, we present some realistic simulations in physiological conditions. Finally, in Section 
\ref{sec:conclusion}, we draw some conclusions.

\section{The mathematical model}
\label{sec:model}
\begin{figure}[t]
	\centering
	{\includegraphics[width=0.8\textwidth]{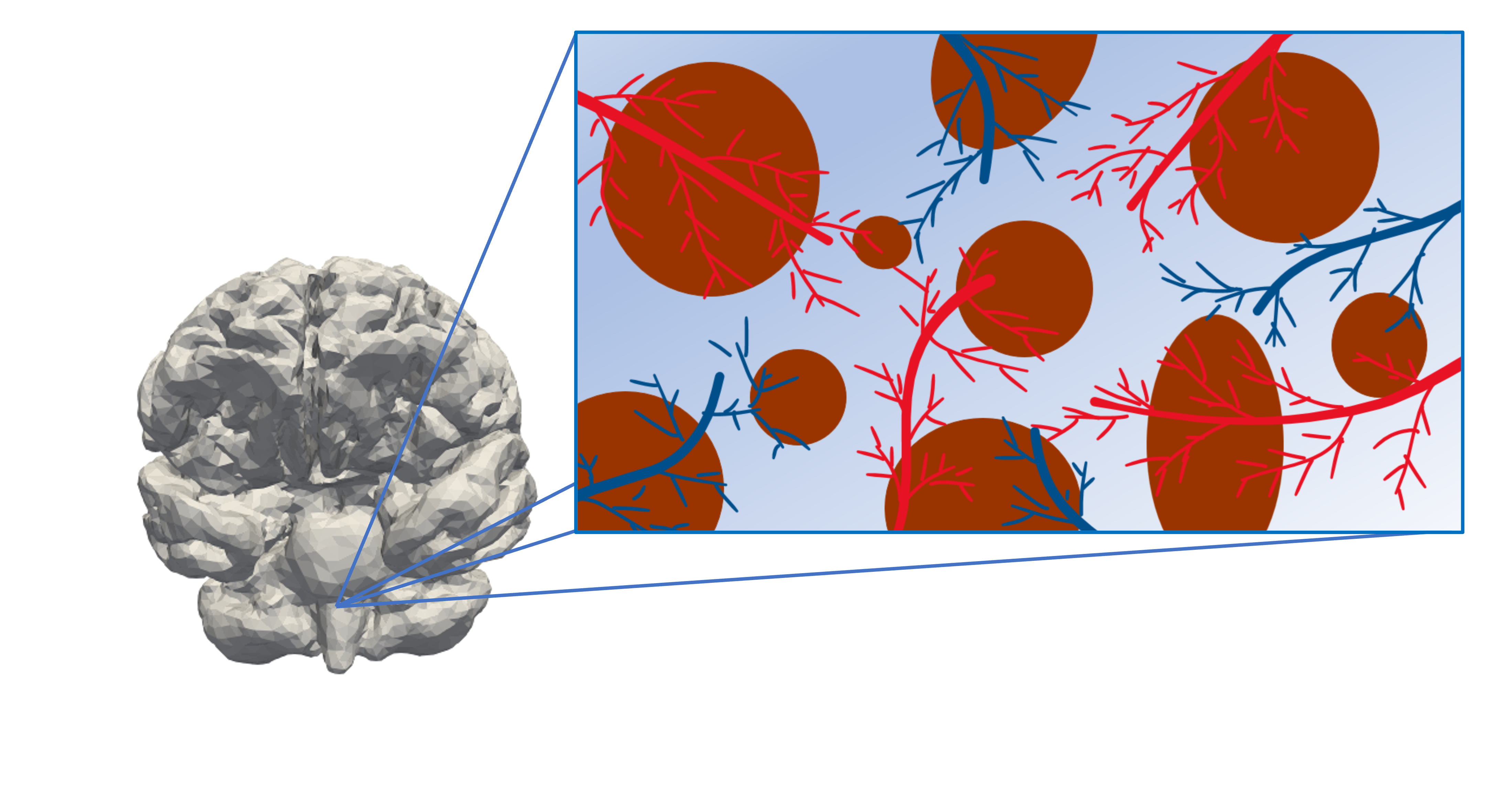}}
	\caption{Example of infinitesimal volume element in which we consider the coexistence of both the solid part (brown) and multiple fluid networks, as we can see in the image: CSF (light-blue), arterial blood (red) and venous blood (blue)}
	\label{fig:infinitesimal}
\end{figure}
\begin{figure}[t]
	\centering
	{\includegraphics[width=\textwidth]{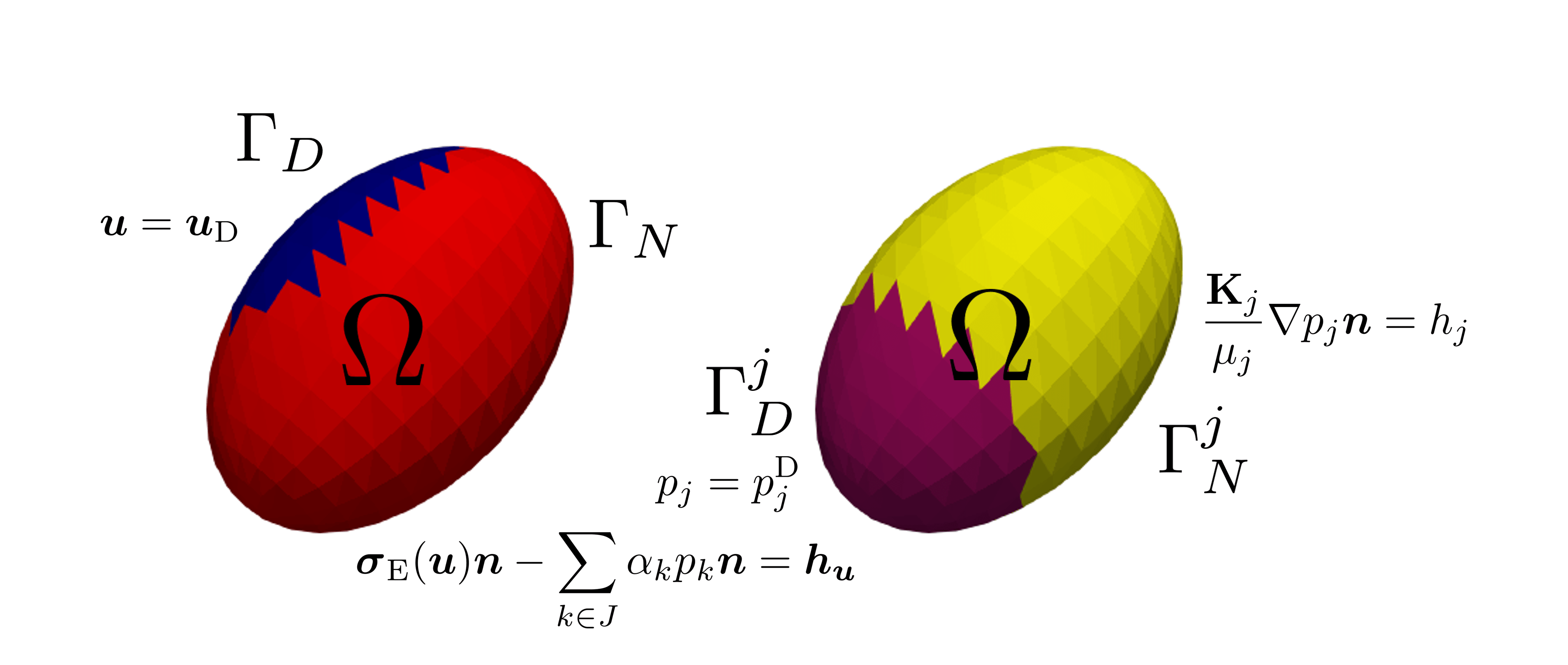}}
	\caption{A domain $\Omega$ with associated boundary conditions for both the displacement $\boldsymbol{u}$ of the tissue and a generic fluid pressure $p_j$ for $j\in J$.}
	\label{fig:ALEdomain}
\end{figure}
In this section, we present the multiple-network poroelasticity system of equations. We consider a given set of labels $J$ such that the magnitude is a number $|J|\in\mathbb{N}$, corresponding to the number of fluid networks. The problem is dependent on time $t\in(0,T]$ and space $\boldsymbol{x}\in\Omega\subset\mathbb{R}^d$ ($d=2,3$). The unknowns of our problem are the displacement $\boldsymbol{u}=\boldsymbol{u}(\boldsymbol{x},t)$ and the network pressures $p_j=p_j(\boldsymbol{x},t)$ for $j\in J$. The problem reads as follows:
\bigskip
\par
Find $\boldsymbol{u}=\boldsymbol{u}(\boldsymbol{x},t)$ and $p_j=p_j(\boldsymbol{x},t)$ such that:
\begin{equation}
\label{eq:initialprob}
\begin{dcases}
    \rho\dfrac{\partial^2\boldsymbol{u}}{\partial t^2}-
    \nabla\cdot\boldsymbol{\sigma}_\mathrm{E}(\boldsymbol{u})+
    \sum_{k\in J} \alpha_k\nabla p_k = \boldsymbol{f},
    & \mathrm{in}\; \Omega\times(0,T],
    \\
    c_j\dfrac{\partial p_j}{\partial t} +
    \nabla\cdot\left(\alpha_j\dfrac{\partial \boldsymbol{u}}{\partial t} -
    \dfrac{\boldsymbol{\mathrm{K}}_j}{\mu_j}\nabla p_j
    \right) +
    \sum_{k \in J} \beta_{jk}(p_j-p_k) + \beta^\mathrm{e}_j p_j=g_j,        
    & \mathrm{in}\; \Omega\times(0,T]\quad\forall j\in J.
\end{dcases}
\end{equation}

In Equation \eqref{eq:initialprob}, we denote the tissue density by $\rho$, the elastic stress tensor by $\boldsymbol{\sigma}_\mathrm{E}$, and the volume force by $\boldsymbol{f}$. Moreover for the $j$-th fluid network, we prescribe a Biot-Willis coefficient $\alpha_j$, a storage coefficient $c_j$, a fluid viscosity $\mu_j$, a permeability tensor $\boldsymbol{\mathrm{K}}_j$, an external coupling coefficient $\beta_j^\mathrm{e}$ and a body force $g_j$. Finally, we have a coupling transfer coefficient $\beta_{jk}$ for each couple of fluid networks $(j,k)\in J\times J$. 
\begin{assumption}[Coefficients' regularity]
In this work, we assume the following regularities for the coefficients and the forcing terms: 
\begin{itemize}
    \item $\rho\in L^\infty(\Omega)$. 
    \item  $\boldsymbol{f}\in L^2((0,T],L^2(\Omega,\mathbb{R}^d))$.
    \item $\alpha_j\in L^\infty(\Omega)$ and $\boldsymbol{\mathrm{K}}_j\in L^\infty(\Omega,\mathbb{R}^{d\times d})$ for any $j\in J$.
    \item $c_j>0$, $\mu_j>0$, and $\beta_j^\mathrm{e}\in L^\infty(\Omega)$ for any $j\in J$.
    \item $g_j\in L^2((0,T],L^2(\Omega))$ for any $j\in J$.
    \item $\beta_{jk}\in L^\infty(\Omega)$ for each couple of fluid networks $(j,k)\in J\times J$.
\end{itemize} 
\end{assumption}
\par
More detailed information about the derivation of this problem can be found in \cite{tullyCerebralWaterTransport2011}. For the purpose of brain poromechanics modelling, we introduce two main modifications to the model:
\begin{itemize}
\item In the derivation we use a static form of the Darcy flow, as in \cite{zienkiewiczDynamicBehaviourSaturated1984}:
\begin{equation}
\boldsymbol{w}_j=-\dfrac{\boldsymbol{\mathrm{K}}_j}{\mu_j}\nabla p_j.
\end{equation}
This is considered a good approximation for the medium-speed phenomena. Indeed, in brain fluid dynamics we do not reach large values of the fluid velocities.
\item We add to the equation a reaction term $\beta^\mathrm{e}_j p_j$ for each fluid network $j$. Indeed, we aim simulating brain perfusion, so we use a diffuse discharge in the venular compartment of the form:
\begin{equation}
\beta^\mathrm{e}_\mathrm{V} (p_\mathrm{V}-\tilde{p}_\mathrm{Veins}),
\end{equation}
considering the large veins pressure $\tilde{p}_\mathrm{Veins}$ and comprising this in the $g_\mathrm{V}$ in the abstract formulation. This mimicks what proposed in the context of heart perfusion \cite{digregorioComputationalModelApplied2021}.
\end{itemize}
It is important to notice that in each infinitesimal volume element we have both the existence of the solid component and of the fluid networks, as represented in Figure \ref{fig:infinitesimal}.
\par
We assume small deformations, that is to consider linear elasticity constitutive relation for the tissue \cite{coussyPoromechanics2004}:
\begin{equation}
\boldsymbol{\sigma}_\mathrm{E}(\boldsymbol{u}) = \mathbb{C}_\mathrm{E}[{\varepsilon}(\boldsymbol{u})] = 2\mu\boldsymbol{\varepsilon}(\boldsymbol{u}) + \lambda(\nabla\cdot\boldsymbol{u}) \boldsymbol{\mathrm{I}},
\end{equation}
where $\mu\in L^\infty(\Omega)$ and $\lambda\in L^\infty(\Omega)$ are the Lamé parameters, $\boldsymbol{\mathrm{I}}$ is the second-order identity tensor, and $\boldsymbol{\varepsilon}(\boldsymbol{u}) = \dfrac{1}{2}(\nabla\boldsymbol{u} + \nabla^\top\boldsymbol{u})$ is the symmetric part of the displacement gradient. Moreover, defined $\mathbb{S}$ as the space of second-order symmetric tensors, $\mathbb{C}_\mathrm{E}:\mathbb{S}\rightarrow\mathbb{S}$ is the fourth order stiffness tensor. This assumption allows us to neglect the differentiation between the actual configuration domain $\Omega_t$ and the reference one $\hat{\Omega}$. For this reason, we consider $\Omega=\hat{\Omega}\simeq\Omega_t$, as in 
Equation \eqref{eq:initialprob}.
\par
We supplement Equation \eqref{eq:initialprob} with suitable boundary and initial conditions. Concerning the initial conditions, due to the second-order time-derivative, we need to impose both a displacement $\boldsymbol{u}_0$ and a velocity $\boldsymbol{v}_0$. Moreover, we need also an initial pressure $p_{j0}$ for each fluid-network $j\in J$. The strong formulation reads:
\begin{equation}
\label{eq:fullprob}
\begin{dcases}
    \rho\dfrac{\partial^2\boldsymbol{u}}{\partial t^2}-
    \nabla\cdot\boldsymbol{\sigma}_\mathrm{E}(\boldsymbol{u})+
    \sum_{k\in J} \alpha_k\nabla p_k = \boldsymbol{f},
    & \mathrm{in}\; \Omega\times(0,T], 
    \\[8pt]
    c_j\dfrac{\partial p_j}{\partial t} +
    \nabla\cdot\left(\alpha_j\dfrac{\partial \boldsymbol{u}}{\partial t} -
    \dfrac{\boldsymbol{\mathrm{K}}_j}{\mu_j}\nabla p_j 
    \right) +
    \sum_{k \in J} \beta_{jk}(p_j-p_k) + \beta^\mathrm{e}_j p_j=g_j,        
    & \mathrm{in}\; \Omega\times(0,T]\quad \forall j\in J,
    \\[8pt]
    \boldsymbol{\sigma}_\mathrm{E}(\boldsymbol{u})\cdot\boldsymbol{n}-\sum_{k\in J}\alpha_k p_k\boldsymbol{n} = \boldsymbol{h_u}, & \mathrm{on}\;\Gamma_N\times(0,T],
    \\[8pt]
    \dfrac{\boldsymbol{\mathrm{K}}_j}{\mu_j}\nabla p_j \boldsymbol{n} = h_j, & \mathrm{on}\;\Gamma_N^j\times(0,T]\quad \forall j\in J,
    \\[8pt]
    \boldsymbol{u} = \boldsymbol{u}_\mathrm{D}, & \mathrm{on}\;\Gamma_D\times(0,T],
    \\[8pt]
    p_j=p^\mathrm{D}_j, & \mathrm{on}\;\Gamma_D^j\times(0,T]\quad \forall j\in J,
    \\[8pt]
    \boldsymbol{u}(0)=\boldsymbol{u}_0, & \mathrm{in}\;\Omega,
    \\[8pt]
    \dfrac{\partial\boldsymbol{u}}{\partial t}(0)=\boldsymbol{v}_0, & \mathrm{in}\;\Omega,
    \\[8pt]
    p_j(0)=p_{j0}, & \mathrm{in}\;\Omega\quad \forall j\in J.
    \\[8pt]
\end{dcases}
\end{equation}
\subsection{Weak formulation}
In order to introduce a numerical approximation to Equation \eqref{eq:fullprob}, we recall the turn to its variational formulation. Let us consider a subset $\Gamma_D\subset\partial\Omega$ with positive measure $|\Gamma_D|>0$, then we define the Sobolev space $ V=H^1_{\Gamma_D}(\Omega,\mathbb{R}^d)$ such that:
\begin{equation}
	H^1_{\Gamma_D}(\Omega,\mathbb{R}^d) := \{\underline{v}\in H^1(\Omega,\mathbb{R}^d):\quad\underline{v}|_{\Gamma_D}=\underline{0}\}.
\end{equation}
Analogously, for a subset $\Gamma_D^j\subset\partial\Omega$ with positive measure $|\Gamma_D^j|>0$ with $j\in J$, we can define the Sobolev space $Q_j=H^1_{\Gamma_D^j}(\Omega)$ such that:
\begin{equation}
	H^1_{\Gamma_D^j}(\Omega) := \{q_j\in H^1(\Omega):\quad q_j|_{\Gamma_D^j}=0\}.
\end{equation}
Moreover, we employ standard definition of scalar product in $L^2(\Omega)$, denoted by $(\cdot,\cdot)_\Omega$. The induced norm is denoted by $||\cdot||_\Omega$.For vector-valued and tensor-valued functions the definition extends componentwise \cite{salsa:EDP}. 
\par
Finally, given $k\in\mathbb{N}$ and an Hilbert space $H$ we use the notation $C^k([0,T],H)$ to denote the space of functions $\boldsymbol{u} = \boldsymbol{u}(\boldsymbol{x},t)$ such that $\boldsymbol{u}$ is $k$-times continuously differentiable with respect to time and for each $t\in[0,T]$, $\boldsymbol{u}(\cdot,t)\in H$, see e.g. in \cite{salsa:EDP}.
\par
The same equation can be also rewritten in an abstract form using the following definitions:
\begin{itemize}
    \item $a:V\times V\rightarrow \mathbb{R}$ is a bilinear form such that:
    \begin{equation}
        a(\boldsymbol{u},\boldsymbol{v}) =  2\mu\left(\boldsymbol{\varepsilon}(\boldsymbol{u}),\boldsymbol{\varepsilon}(\boldsymbol{v})\right)_\Omega +\lambda\left(\nabla\cdot\boldsymbol{u},\nabla\cdot\boldsymbol{v}\right)_\Omega \qquad \forall\boldsymbol{u}, \boldsymbol{v}\in V
    \end{equation}

    \item $b_j:Q_j\times V\rightarrow \mathbb{R}$ is a bilinear form such that:
    \begin{equation}
        b_j(q_j,\boldsymbol{v}) = \alpha_j \left(q_j,\nabla\cdot\boldsymbol{v} \right)_\Omega \qquad \forall q_j\in Q_j\quad \forall \boldsymbol{v}\in V
    \end{equation}
    
    \item $F:V\rightarrow \mathbb{R}$ is a linear functional such that:
    \begin{equation}
        F(\boldsymbol{v}) = (\boldsymbol{f},\boldsymbol{v})_\Omega + 
     \left(\boldsymbol{h}_{\boldsymbol{u}},\boldsymbol{v}\right)_{\Gamma_N} \qquad \boldsymbol{v}\in V
    \end{equation}
    \item $s_j : Q_j \times Q_j \rightarrow \mathbb{R}$ is a bilinear form such that:
    \begin{equation}
        s_j(p_j,q_j) = \left(\dfrac{\boldsymbol{\mathrm{K}}_j}{\mu_j}\nabla p_j ,\nabla q_j\right)_\Omega \quad \forall p_j,q_j\in Q_j,
    \end{equation}
    \item $C_j: \left(\displaystyle\bigtimes_{k\in J} Q_k\right)\times Q_j\rightarrow \mathbb{R}$ is a bilinear form such that:
    \begin{equation}
        C_j\left((p_k)_{k\in J},q_j\right) = \sum_{k \in J}\left(\beta_{jk}(p_j-p_k),q_j\right)_\Omega + (\beta_{j}^\mathrm{e} p_j,q_j)_\Omega  
\end{equation}
    \item $G_j : Q_j \rightarrow \mathbb{R}$ is a linear functional such that:
    \begin{equation}
        G_j(q_j) = (g_j,q_j)_\Omega + (h_j,q_j)_{\Gamma_N^j} \quad\forall q_j\in Q_j.
    \end{equation}
\end{itemize}

The weak formulation of problem \eqref{eq:fullprob} reads:
\par
\bigskip
Find $\boldsymbol{u}(t)\in V$ and $q_j(t) \in Q_j$ with $j\in J$ such that $\forall t>0$:
\begin{equation}
\label{eq:weakform}
\begin{dcases}
     \rho\left(\dfrac{\partial^2\boldsymbol{u}(t)}{\partial t^2},\boldsymbol{v}\right)_\Omega + a(\boldsymbol{u}(t),\boldsymbol{v}) - \sum_{k \in J} b_k(p_k(t),\boldsymbol{v}) = F(\boldsymbol{v}),
     & \forall\boldsymbol{v}\in V, \\[8pt]
     c_j\left(\dfrac{\partial p_j}{\partial t},q_j\right)_{\Omega} + b_j\left(q_j,
    \dfrac{\partial \boldsymbol{u}}{\partial t}\right)
    + s_j(p_j,q_j) +   C_j\left((p_k)_{k\in J},q_j\right) = G_j(q_j), & \forall q_j\in Q_j\quad j\in J, \\[8pt]
    \boldsymbol{u}(0)=\boldsymbol{u}_0, & \mathrm{in}\;\Omega, \\[8pt]
    \dfrac{\partial\boldsymbol{u}}{\partial t}(0)=\boldsymbol{v}_0, & \mathrm{in}\;\Omega, \\[8pt]
    p_j(0)=p_{j0}, & \mathrm{in}\;\Omega\quad j\in J,\\[8pt]
    \boldsymbol{u}(t) = \boldsymbol{u}_\mathrm{D}(t), & \mathrm{on}\;\Gamma_D,\\[8pt]
   q_j(t)=q^\mathrm{D}_j(t), & \mathrm{on}\;\Gamma_D^j \quad j\in J.
\end{dcases}
\end{equation}
The complete derivation of this formulation is reported in Appendix \ref{sec:wfder}.
\section{PolyDG semi-discrete formulation}
\label{sec:polydg}
Let us introduce a polytopic mesh partition $\partition$ of the domain $\Omega$ made of polygonal/polyhedral elements $K$ such that:
\begin{equation*}
\forall K_i,K_j \in \partition\qquad |K_i\cap K_j| = 0\quad \mathrm{if} \quad i\neq j \qquad \bigcup_{j} K_j = \Omega
\end{equation*}
where we for each element $K\in \partition$, we denote by $|K|$ the measure of the element and by $h_K$ its diameter. We set $h\in \max_{K\in\partition} h_K$.
\par
Then we can define the interface as the intersection of the $(d-1)-$dimensional facets of two neighbouring elements. We distinguish two cases:
\begin{itemize}
    \item case $d=3$, in which the interface consists in a generic polygon, we further assume that we can decompose each interface into triangles; we denote the set of these triangles with $\faces$;
    \item case $d=2$, in which the interfaces are always line segments; then we denote such a set of segments with $\faces$.
\end{itemize}
It is now useful to subdivide the set into the union of interior faces $\facesinternal$ and $\facesboundary$ exterior faces lying on the boundary of the domain $\partial\Omega$:
\begin{equation*}
    \faces = \facesinternal \cup \facesboundary.
\end{equation*}
Moreover the boundary faces set can be split according to the type of imposed boundary condition of the tissue displacement:
\begin{equation*}
    \facesboundary = \facesD \cup \facesN,
\end{equation*}
where $\facesD$ and $\facesN$ are the boundary faces contained in $\Gamma_D$ and $\Gamma_N$, respectively.  Implicit in this decomposition, there is the assumption that $\partition$ is aligned with $\Gamma_D$ and $\Gamma_N$, i.e. any $F\in\facesboundary$ is contained in either $\Gamma_D$ and $\Gamma_N$. The same splitting can be done according to the type of imposed boundary condition of the generic $j$-th fluid network:
\begin{equation*}
    \facesboundary = \facesDj \cup \facesNj,
\end{equation*}
where $\facesDj$ and $\facesNj$ are the boundary faces contained in $\Gamma_D^j$ and $\Gamma_N^j$, respectively. Implicit in this decomposition, there is the assumption that $\partition$ is aligned with $\Gamma_D^j$ and $\Gamma_N^j$, i.e. any $F\in\facesboundary$ is contained in either $\Gamma_D^j$ and $\Gamma_N^j$. 
\par
Let us define $\mathbb{P}_{s}(K)$ as the space of polynomials of degree $s$ over a mesh element $K$. 
Then we can introduce the following discontinuous finite element spaces:
\begin{equation*}
    \Qh = \{q\in L^2(\Omega):\quad q|_K\in\mathbb{P}_{q}(K)\quad\forall K\in\partition\},
\end{equation*}
\begin{equation*}
    \Vh = \{\boldsymbol{w}\in L^2(\Omega;\mathbb{R}^d):\quad\boldsymbol{w}|_K\in[\mathbb{P}_{p}(K)]^d\quad\forall K\in\partition\},
\end{equation*}
where $p\geq 1$ and $q\geq 1$ are polynomial orders, which can be different in principle. 
\par
Finally, we introduce some assumptions on $\partition$.
\begin{definition}[Polytopic regular mesh]
Let $\partition$ be a mesh, we say it is polytopic regular if:
\begin{equation*}
    \forall K\in\partition\quad \exists \{S_K^F\}_{F\subset\partial K} \;\mathrm{such}\;\mathrm{that} \quad \forall F\subset\partial K \quad \overline{F}=\partial\overline{K}\cap\overline{S^F_K}\;\mathrm{and}\; h_K \lesssim d|S_K^F|\;|F|^{-1},
\end{equation*}
where $\{S_K^F\}_{F\subset\partial K}$ is a set of non-overlapping $d$-dimensional simplices contained in $K$ and $h_K$ is the diameter of the element $K$.
\end{definition}
We remark that the union of simplices $\{S_F\}_{F\subset\partial K}$ does not have to cover, in general, the whole element $K$, that is $\bigcup_\{\bar{S}_F\}_{F\subset\partial K}\subset \bar{K}$.
\begin{assumption}
The mesh sequence $\{\partition\}_h$ satisfies the following properties:
\begin{enumerate}
    \item $\{\partition\}_{h>0}$ is uniformly polytopic-regular
    \item For each $\partition \in \{\partition\}_h$ there exists a shape-regular, simplicial covering $\hat{\partition}$ of $\partition$ such that for each pair $K\in\partition$ and $\hat{K}\in\hat{\partition}$ with $\hat{K}\subset K$ it holds:
    \begin{enumerate}
        \item $h_K\lesssim h_{\hat{K}}$;
        \item $\underset{K\in\partition}{\max}|K'|\lesssim 1$ where $K'\in\partition: K'\cap \hat{K} \neq 0, \hat{K}\in\hat{\partition}, K\subset \hat{K}$.
    \end{enumerate}
    \item A local bounded variation property holds for the local mesh sizes:
    \begin{equation*}
        \forall F\in\faces F\subset\partial K_1\cap\partial K_2 \quad K_1,K_2\in\partition \Rightarrow h_{K_1}\lesssim h_{K_2} \lesssim h_{K_1}
    \end{equation*}
    where the hidden constants are independent of both discretization parameters and number of faces of $K_1$ and $K_2$.
\end{enumerate}
\end{assumption}
\par
We next introduce the so-called trace operators \cite{arnoldUnifiedAnalysisDiscontinuous2001}. Let $F\in\facesinternal$ be a face shared by the elements $K^\pm$. Let $\boldsymbol{n}^\pm$ by the unit normal vector on face $F$ pointing exterior to $K^\pm$, respectively. Then, assuming sufficiently regular scalar-valued functions $q$, vector-valued functions $\boldsymbol{v}$ and tensor-values functions $\boldsymbol{\tau}$, we can define:
\begin{itemize}
    \item the average operator $\averagel{\cdot}\averager$ on $F\in \facesinternal$: 
    \begin{equation}
        \averagel{q}\averager = \dfrac{1}{2} (q^+ + q^-), \qquad \averagel{\boldsymbol{v}}\averager = \dfrac{1}{2} (\boldsymbol{v}^+ + \boldsymbol{v}^-), \qquad
        \averagel{\boldsymbol{\tau}}\averager  = \dfrac{1}{2}(\boldsymbol{\tau}^+ + \boldsymbol{\tau}^-),
    \end{equation}
    \item the jump operator $\jumpl{\cdot}\jumpr$ on $F\in \facesinternal$: 
    \begin{equation}
        \jumpl{q}\jumpr = q^+\boldsymbol{n}^+ + q^-\boldsymbol{n}^-,
        \qquad
        \jumpl{\boldsymbol{v}}\jumpr = \boldsymbol{v}^+\cdot\boldsymbol{n}^+ + \boldsymbol{v}^-\cdot\boldsymbol{n}^-,
        \qquad
        \jumpl{\boldsymbol{\tau}}\jumpr  = \boldsymbol{\tau}^+\boldsymbol{n}^+ + \boldsymbol{\tau}^-\boldsymbol{n}^-,
    \end{equation}
    \item the jump operator $\jjumpl\cdot\jjumpr$ on $F\in \facesinternal$ for a vector-valued function:
        \begin{equation}
            \jjumpl\boldsymbol{v}\jjumpr = \dfrac{1}{2}(\boldsymbol{v}^+\otimes\boldsymbol{n}^++\boldsymbol{n}^+\otimes\boldsymbol{v}^+) + \dfrac{1}{2}(\boldsymbol{v}^-\otimes\boldsymbol{n}^-+\boldsymbol{n}^-\otimes\boldsymbol{v}^-),
        \end{equation}
        where the result is a tensor in $\mathbb{R}^{d\times d}_\mathrm{sym}$.
\end{itemize}
In these relations we are using the superscripts $\pm$ on the functions, to denote the traces of the functions on $F$ taken within the of interior to $K^\pm$.
\par
In the same way, we can define analogous operators on the face $F\in\facesboundary$ associated to the cell $K\in\partition$ with $\boldsymbol{n}$ outward unit normal on $\partial\Omega$:
\begin{itemize}
    \item the average operator $\averagel{\cdot}\averager$ on $F\in\facesboundary$: 
    \begin{equation}
        \averagel{q}\averager = q, \qquad \averagel{\boldsymbol{v}}\averager = \boldsymbol{v}, \qquad
        \averagel{\boldsymbol{\tau}}\averager  = \boldsymbol{\tau},
    \end{equation}
    \item the standard jump operator $\jumpl{\cdot}\jumpr$ on $F\in\facesboundary$ which does not belong to a Dirichlet boundary: 
    \begin{equation}
        \jumpl{q}\jumpr = q\boldsymbol{n},
        \qquad
        \jumpl{\boldsymbol{v}}\jumpr = \boldsymbol{v}\cdot\boldsymbol{n},
        \qquad
        \jumpl{\boldsymbol{\tau}}\jumpr  = \boldsymbol{\tau}\boldsymbol{n},
    \end{equation}
    \item the jump operator $\jjumpl{\cdot}\jjumpr$ on $F\in\facesboundary$ which belongs to a Dirichlet boundary, with Dirichlet conditions $g$, $\boldsymbol{g}$ and $\boldsymbol{\gamma}$: 
    \begin{equation}
        \jumpl{q}\jumpr = (q-g)\boldsymbol{n},
        \qquad
        \jumpl{\boldsymbol{v}}\jumpr = (\boldsymbol{v}-\boldsymbol{g})\cdot\boldsymbol{n},
        \qquad
        \jumpl{\boldsymbol{\tau}}\jumpr  = (\boldsymbol{\tau}-\boldsymbol{\gamma})\boldsymbol{n},
    \end{equation}
    \item the jump operator on $F\in\facesboundary$ for a vector-valued function which does not belong to a Dirichlet boundary:
        \begin{equation}
            \jjumpl\boldsymbol{v}\jjumpr = \dfrac{1}{2}(\boldsymbol{v}\otimes\boldsymbol{n}+\boldsymbol{n}\otimes\boldsymbol{v}).
        \end{equation}
    \item the jump operator on $F\in\facesboundary$ for a vector-valued function which belongs to a Dirichlet boundary, with Dirichlet condition $\boldsymbol{g}$:
    \begin{equation}
        \jjumpl\boldsymbol{v}\jjumpr = \dfrac{1}{2}((\boldsymbol{v}-\boldsymbol{g})\otimes\boldsymbol{n}+\boldsymbol{n}\otimes(\boldsymbol{v}-\boldsymbol{g})).
    \end{equation}
\end{itemize}
We recall the following identity will be useful in the method derivation:
\begin{equation}
    \jumpl q \boldsymbol{v} \jumpr = \jumpl \boldsymbol{v} \jumpr \averagel q \averager + \averagel \boldsymbol{v} \averager \cdot \jumpl q \jumpr, \qquad \forall F \in \facesinternal.
\end{equation}
Finally, we remark also the following identities \cite{arnoldInteriorPenaltyFinite1982, antoniettiblanca} for $\boldsymbol{\tau}\in \mathbf{L}^2(\Omega,\mathbb{R}^{d\times d}_\mathrm{sym})$, $\boldsymbol{v}\in \mathbf{H}^1(\Omega,\mathbb{R}^{d})$, and $q\in H^1(\Omega)$:
\begin{equation}
    \sum_{K\in\partition}\int_{\partial K} q\boldsymbol{v}\cdot\boldsymbol{n}_K = \sum_{F\in\faces} \int_F \averagel\boldsymbol{v}\averager\cdot\jumpl q\jumpr
    +  \sum_{F\in\facesinternal} \int_F \averagel q\averager\cdot\jumpl \boldsymbol{v}\jumpr,
\end{equation}
\begin{equation}
    \sum_{K\in\partition}\int_{\partial K} \boldsymbol{v}\cdot(\boldsymbol{\tau}\boldsymbol{n}_K) = 
    \sum_{K\in\partition}\int_{\partial K} \boldsymbol{\tau}:(\boldsymbol{v}\otimes\boldsymbol{n}_K) = 
    \sum_{F\in\faces} \int_F \averagel \boldsymbol{\tau}\averager:\jjumpl \boldsymbol{v}\jjumpr
    +  \sum_{F\in\facesinternal} \int_F \averagel\boldsymbol{v}\averager\cdot\jumpl \boldsymbol{\tau}\jumpr,
\end{equation}
where $\boldsymbol{n}_K$ is the outward normal unit vector to the cell $K$.
\subsection{Semi-discrete formulation}
To construct the semi-discrete formulation, we define the following penalization functions $\eta:\faces\rightarrow\mathbb{R}_+$ and $\zeta_j:\faces\rightarrow\mathbb{R}_+$ for each $j\in J$, which are face-wise defined as:
\begin{equation}
    \eta = \eta_0 \tilde{\mathbb{C}}_\mathrm{E}^K
    \begin{cases}
         \dfrac{p^2}{\{h\}_\mathrm{H}},  & \mathrm{on}\; F\in\facesinternal\\
         \dfrac{p^2}{h},                 & \mathrm{on}\; F\in\facesD
    \end{cases}
    \qquad
    \zeta_j = z_j \dfrac{k_j^K}{\sqrt{\mu_j}}
    \begin{cases}
         \dfrac{q^2}{\{h\}_\mathrm{H}},  & \mathrm{on}\; F\in\facesinternal\\
         \dfrac{q^2}{h},                 & \mathrm{on}\; F\in\facesboundary
    \end{cases},
\end{equation}
where we are considering the harmonic average operator $\{\cdot\}_\mathrm{H}$ on $K^\pm$, $\tilde{\mathbb{C}}_\mathrm{E}^K = \Big|\Big|\sqrt{\mathbb{C}_\mathrm{E}}|_K\Big|\Big|_2^2$ and $k_j^K = ||\sqrt{\mathbf{K}_j|_K}||_2^2$ for any $K\in\partition$\footnote{In this context $||\cdot||_2$ is the operator norm induced by the $L^2$-norm in the space of symmetric second order tensors.} and $\eta_0$ and $z_j$ are parameters at our disposal (to be chosen large enough). The parameters $z_j$ require to be chosen appropriately in particular for small values of $k_j^K$, which are typical in applications. Moreover, we need to define the following bilinear forms:
\begin{itemize}
    \item $\mathcal{A}_\mathrm{E}:\Vh\times \Vh\rightarrow \mathbb{R}$ is a bilinear form such that:
\begin{equation}
    \mathcal{A}_\mathrm{E}(\boldsymbol{u},\boldsymbol{v}) = \int_{\Omega}\boldsymbol{\sigma}_\mathrm{E}(\boldsymbol{u}):\nabla_h\boldsymbol{v}+\sum_{F\in\facesinternal\cup\facesD}\int_{F}\left(\eta \jjumpl\boldsymbol{u}\jjumpr : \jjumpl\boldsymbol{v}\jjumpr-\averagel\boldsymbol{\sigma}_\mathrm{E}(\boldsymbol{u}_h)\averager : \jjumpl\boldsymbol{v}_h\jjumpr -  \jjumpl\boldsymbol{u}_h\jjumpr : \averagel\boldsymbol{\sigma}_\mathrm{E}(\boldsymbol{v}_h)\averager\right)\mathrm{d}\sigma,
\end{equation}
for all $\boldsymbol{u},\boldsymbol{v}\in\Vh$.

    \item $\mathcal{B}_j:\Qh\times \Vh\rightarrow \mathbb{R}$ is a bilinear form for any $j\in J$ such that:
\begin{equation}
    \mathcal{B}_j(p_j,\boldsymbol{v}) = 
    \int_{\Omega} \alpha_j p_{j}(\nabla_h\cdot\boldsymbol{v})
    - \sum_{F\in\facesinternal\cup\facesDj}\int_{F}\alpha_j \averagel p_{jh}\mathrm{\mathbf{I}}\averager:\jjumpl\boldsymbol{v}_h\jjumpr\mathrm{d}\sigma 
    \qquad\forall p_j\in\Qh\;\forall \boldsymbol{v}\in\Vh.
\end{equation}
    
    \item $\mathcal{A}_{\mathrm{P}_j} : \Qh \times \Qh \rightarrow \mathbb{R}$ is a bilinear form such that:
\begin{equation}
    \begin{split}
    \mathcal{A}_{\mathrm{P}_j}(p_j,q_j) = &
    \int_{\Omega} \dfrac{\boldsymbol{\mathrm{K}}_j}{\mu_j}\nabla_h p_j\cdot\nabla_h q_j - \sum_{F\in\facesinternal\cup\facesDj} \int_{F} \dfrac{1}{\mu_j}\averagel\boldsymbol{\mathrm{K}}_j\nabla_h p_j\averager\cdot\jumpl q_j\jumpr  + \\ - & \sum_{F\in\facesinternal\cup\facesDj} \int_{F} \dfrac{1}{\mu_j}\averagel\boldsymbol{\mathrm{K}}_j\nabla_h q_j\averager\cdot\jumpl p_j\jumpr  + \sum_{F\in\facesinternal\cup\facesDj} \int_{F} \zeta_j \jumpl p_j\jumpr\cdot\jumpl q_j\jumpr \qquad p_j,q_j\in\Qh.
    \end{split}
\end{equation}
\end{itemize}
By exploiting the definitions of the bilinear forms, we obtain the following semi-discrete PolyDG formulation.
\par
\bigskip
Find $\boldsymbol{u}_h(t)\in \Vh$ and $p_{jh}(t) \in \Qh$ with $j\in J$ such that $\forall t>0$:
\begin{equation}
\begin{dcases}
     \rho\left(\ddot{\boldsymbol{u}}_h(t),\boldsymbol{v}_h\right)_\Omega + \mathcal{A}_\mathrm{E}(\boldsymbol{u}_h(t),\boldsymbol{v}_h) - \sum_{k \in J} \mathcal{B}_k(p_{kh}(t),\boldsymbol{v}_h) = F(\boldsymbol{v}_h),
     & \forall\boldsymbol{v_h}\in \Vh \\[8pt]
    c_j\left(\dot{p}_{jh}(t),q_{jh}\right)_{\Omega} + \mathcal{B}_j\left(q_{jh},
    \dot{\boldsymbol{u}}_h(t)\right) + 
    \mathcal{A}_{\mathrm{P}_j}(p_{jh}(t),q_{jh}) +  C_j\left((p_{kh})_{k\in J},q_{jh}\right)  = G_j(q_{jh}), & \forall q_{jh}\in\Qh  \\[8pt]
    \boldsymbol{u}_h(0)=\boldsymbol{u}_{0h}, & \mathrm{in}\;\Omega_h \\[8pt]
    \dot{\boldsymbol{u}}_h(0)=\boldsymbol{v}_{0h}, & \mathrm{in}\;\Omega_h \\[8pt]
    p_{jh}(0)=p_{j0h}, & \mathrm{in}\;\Omega_h\\[8pt]
    \boldsymbol{u}_h(t) = \boldsymbol{u}^\mathrm{D}_h(t), & \mathrm{on}\;\Gamma_D\\[8pt]
   q_{jh}(t)=q^\mathrm{D}_{jh}(t), & \mathrm{on}\;\Gamma_D^j
\end{dcases}
\label{eq:DGFormulation}
\end{equation}
The complete derivation of this formulation is reported in Appendix \ref{sec:PolyDGder}. Summing up the weak formulations we arrive to the following equivalent equation, we will use in the analysis:
\begin{equation}
\label{eq:summed}
    \begin{split}
        \rho\left(\ddot{\boldsymbol{u}}_h(t),\boldsymbol{v}_h\right)_\Omega + & \mathcal{A}_\mathrm{E}(\boldsymbol{u}_h(t),\boldsymbol{v}_h) + \sum_{k \in J}\Bigg( -\mathcal{B}_k(p_{kh}(t),\boldsymbol{v}_h) + c_k\left(\dot{p}_{kh}(t),q_{kh}\right)_{\Omega} + \mathcal{A}_{\mathrm{P}_k}(p_{kh}(t),q_{kh}) \\ + &  \mathcal{B}_k\left(q_{kh},
    \dot{\boldsymbol{u}}_h(t)\right) 
    +  C_k\left((p_{jh})_{j\in J},q_{kh}\right)  \Bigg) = F(\boldsymbol{v}_h) + \sum_{k \in J} G_k(q_{kh})
     \quad \forall\boldsymbol{v_h}\in \Vh \;\forall q_{kh}\in\Qh.
    \end{split}
\end{equation}
\section{Stability analysis of the semi-discrete formulation}
\label{sec:stability}
To carry out a complete stability analysis of the problem \eqref{eq:summed}, we introduce the following broken Sobolev spaces for an integer $r\geq1$:
\begin{equation*}
    H^r(\mathcal{T}_h) = \{v_h\in L^2(\Omega): v_h|_K\in H^r(K)\quad \forall K\in\mathcal{T}_h\},
\end{equation*}
\begin{equation*}
        \mathbf{H}^r(\mathcal{T}_h;\mathbb{R}^d) = \{v_h\in L^2(\Omega;\mathbb{R}^d): v_h|_K\in H^r(K;\mathbb{R}^d)\quad \forall K\in\mathcal{T}_h\}.
\end{equation*}
Moreover, we introduce the shorthand notation for the $L^2$-norm $||\cdot||:=||\cdot||_{L^2(\Omega)}$ and for the $L^2$-norm on a set of faces $\mathcal{F}$ as $||\cdot||_\mathcal{F}=\left(\sum_{F\in\mathcal{F}}||\cdot||_{L^2(F)}\right)^{1/2}$.
\par
These norms can be used to define the following DG-norms:
\begin{equation}
    ||p||_{\mathrm{DG,P}_j} = \Big|\Big|\sqrt{\dfrac{\mathbf{K}_j}{\mu_j}}\nabla_h p \Big|\Big| + ||\sqrt{\zeta_j}\jumpl p\jumpr||_{\mathrm{L}^2(\facesinternal\cup\facesDj)} \qquad \forall p\in H^1(\partition)
\end{equation}
\begin{equation}
    ||\boldsymbol{v}||_\mathrm{DG,E} = \Big|\Big|\sqrt{\mathbb{C}_\mathrm{E}}[\boldsymbol{\varepsilon}_h(\boldsymbol{v})] \Big|\Big| + ||\sqrt{\eta}\jjumpl\boldsymbol{v}\jjumpr||_{\mathrm{L}^2(\facesinternal\cup\facesD)}\qquad \forall \boldsymbol{v}\in\mathbf{H}^1(\partition;\mathbb{R}^d)
\end{equation}
For the analysis, we need to prove some continuity and coercivity properties of the bilinear forms.
\begin{proposition}
Let Assumption 2 be satisfied, then the bilinear forms $\mathcal{A}_\mathrm{E}(\cdot,\cdot)$ and $\mathcal{A}_{\mathrm{P}_j}(\cdot,\cdot)$ are continuous:
\begin{equation}
    |\mathcal{A}_\mathrm{E}(\boldsymbol{v}_h,\boldsymbol{w}_h)| \lesssim ||\boldsymbol{v}_h||_\mathrm{DG,E} ||\boldsymbol{w}_h||_\mathrm{DG,E} \qquad \forall \boldsymbol{v}_h,\boldsymbol{w}_h\in \mathbf{V}_h^\mathrm{DG},
\end{equation}
\begin{equation}
    |\mathcal{A}_{\mathrm{P}_j}(p_{jh},q_{jh})| \lesssim ||p_{jh}||_{\mathrm{DG,P}_j}||q_{jh}||_{\mathrm{DG,P}_j} \qquad \forall p_{jh},q_{jh}\in Q_h^\mathrm{DG}\qquad \forall j \in J,
\end{equation}
and coercive:
\begin{equation}
    \mathcal{A}_\mathrm{E}(\boldsymbol{v}_h,\boldsymbol{v}_h) \gtrsim ||\boldsymbol{v}_h||_\mathrm{DG,E}^2 \qquad \forall \boldsymbol{v}_h\in \Vh,
\end{equation}
\begin{equation}
    \mathcal{A}_{\mathrm{P}_j}(p_{jh},p_{jh}) \gtrsim ||p_{jh}||_{\mathrm{DG,P}_j}^2 \qquad \forall p_{jh}\in Q_h^\mathrm{DG}\qquad \forall j \in J,
\end{equation}
provided that the penalty parameters $eta$ and $\zeta_j$ for any $j\in J$ are chosen large enough.
\end{proposition}
The proof of these properties can be found in \cite{antoniettiHighorderDiscontinuousGalerkin2018}.
\begin{proposition}
Let Assumption 2 be satisfied. 
The bilinear form $\mathcal{B}_j$ is also continuous:
\begin{equation}
    |\mathcal{B}_j(q_{jh},\boldsymbol{v}_h)|\lesssim||\boldsymbol{v}_h||_{\mathrm{DG,E}}||q_{jh}|| \quad \forall \boldsymbol{v}_h,\in \Vh \quad \forall q_{jh}\in\Qh 
\end{equation}
\end{proposition}
The proof of these properties can be found in \cite{antoniettiStabilityAnalysisPolytopic2021}. 
\begin{proposition}
Let Assumption 2 be satisfied, then:
\begin{equation}
    \left|\sum_{j \in J} C_j\left((p_{kh})_{k\in J},q_{jh}\right)\right|\lesssim \sum_{k \in J}\sum_{j \in J} ||p_{kh}||\; ||q_{jh}|| \quad \forall p_{kh},q_{kh}\in\Qh,
\end{equation}
\begin{equation}
    \sum_{j \in J} C_j\left((p_{kh})_{k\in J},p_{jh}\right))\gtrsim \sum_{j \in J} ||\sqrt{\beta_j^\mathrm{e}}p_{jh}||^2 \quad \forall p_{jh}\in\Qh.
\end{equation}
\end{proposition}
\begin{proof}
First of all, to simplify the computations let us introduce the following quantity:
\begin{equation}
\label{eq:linfty}
    \mathbb{B}=\max\left\{\underset{j,k\in J}{\max}\left\{||\beta_{jk}||_{L^\infty(\Omega)}\right\},\underset{j\in J}{\max}\left\{||\beta_{j}^\mathrm{e}||_{L^\infty(\Omega)}\right\}\right\}
\end{equation}
The proof of the continuity trivially derives from the application of triangular inequality and H\"older inequality, using relation \eqref{eq:linfty}:
\begin{equation*}
\begin{split}
    \left|\sum_{j \in J} C_j\left((p_{kh})_{k\in J},q_j\right)\right| \leq & \sum_{k \in J}\sum_{j \in J}|(\beta_{kj}p_{kh},q_{kh})_\Omega|+\sum_{k \in J}\sum_{j \in J}|(\beta_{kj}p_{jh},q_{kh})_\Omega| + \sum_{j \in J}|(\beta_{j}^\mathrm{e} p_{jh},q_{jh})_\Omega| \leq \\
    \leq & \sum_{k \in J}\sum_{j \in J}\left(2\mathbb{B}||p_{kh}||\,||q_{kh}||+\mathbb{B}||p_{jh}||\,||q_{kh}|| \right)\lesssim  
    \sum_{k \in J}\sum_{j \in J}||p_{jh}||\,||q_{kh}||
\end{split}
\end{equation*}
In the last step, we are observing that in the second sum we are also controlling the case $j=k$.
\par
To prove the coercivity, we introduce the definition of $\tilde{\beta}_j = \sum_{k \in J}\beta_{kj}+\beta_j^\mathrm{e} = \sum_{k \in J}\beta_{jk}+\beta_j^\mathrm{e}>0$. Then we proceed as:
\begin{equation*}
\begin{split}
       \sum_{j \in J} C_j&\left((p_{kh})_{k\in J},p_{jh}\right)) =  \sum_{j \in J}\sum_{k \in J}(\beta_{jk}(p_{jh}-p_{kh}),p_{jh})_\Omega + \sum_{j \in J} (\beta_j^\mathrm{e}p_{jh},p_{jh})_\Omega = \\ 
       = &  \sum_{j \in J}\sum_{k \in J}||\sqrt{\beta_{jk}}\;p_{jh}||^2 + \sum_{j \in J}||\sqrt{\beta_j^\mathrm{e}}\;p_{jh}||^2 - \sum_{j \in J}\sum_{k \in J}(\beta_{jk}p_{kh},p_{jh})_\Omega \geq \\ 
       \geq &    \sum_{j \in J}\sum_{k \in J}||\sqrt{\beta_{jk}}\;p_{jh}||^2 + \sum_{j \in J}||\sqrt{\beta_j^\mathrm{e}}\;p_{jh}||^2 - \sum_{j \in J}\sum_{k \in J}|(\beta_{jk}p_{jh},p_{kh})_\Omega| \geq \qquad\quad\mathrm{H\ddot{o}lder\,inequality} \\
        \geq &    \sum_{j \in J}\sum_{k \in J}||\sqrt{\beta_{jk}}\;p_{jh}||^2 + \sum_{j \in J}||\sqrt{\beta_j^\mathrm{e}}\;p_{jh}||^2 - \sum_{j \in J}\sum_{k \in J}||\sqrt{\beta_{jk}}p_{jh}||\,||\sqrt{\beta_{kj}}p_{kh}|| \geq \;\mathrm{Young\,inequality}\\
        \geq &    \sum_{j \in J}\sum_{k \in J}||\sqrt{\beta_{jk}}\;p_{jh}||^2 + \sum_{j \in J}||\sqrt{\beta_j^\mathrm{e}}\;p_{jh}||^2 - \dfrac{1}{2}\sum_{k \in J}\sum_{j \in J}||\sqrt{\beta_{jk}}p_{jh}||^2- \dfrac{1}{2}\sum_{k \in J}\sum_{j \in J}||\sqrt{\beta_{kj}}p_{kh}||^2 \geq \\
        \geq &  \sum_{j \in J}||\sqrt{\beta_j^\mathrm{e}}\;p_{jh}||^2,
\end{split}
\end{equation*}
and the thesis follows.
\end{proof}

\subsection{Stability estimate} 
For the sake of simplicity, we assume homogeneous boundary conditions, both on Neumann and Dirichlet boundaries, i.e. $\boldsymbol{u}_\mathrm{D}=\boldsymbol{0}$, $\boldsymbol{h}_\mathrm{u}=\boldsymbol{0}$, $h_j = 0$ and $p_j^\mathrm{D}=0$ for any $j\in J$. 
\begin{definition}
Let us define the following energy norm:
\begin{equation}
\begin{split}
    || (\boldsymbol{u}_h,&(p_{kh})_{k\in J})(t)||_\varepsilon^2 = \\ = &
 ||\sqrt{\rho}\dot{\boldsymbol{u}}_h(t)||^2 + ||\boldsymbol{u}_h(t)||_\mathrm{DG,E}^2  + \sum_{k \in J}\left( ||\sqrt{c_k} p_{kh}(t)||^2 + \int_0^t \left(||p_{kh}(s)||_{\mathrm{DG,P}_k}^2 + ||\sqrt{\beta_k^\mathrm{e}}\;p_{kh}(s)||^2\right) \mathrm{d}s \right)
\end{split}
\end{equation}
\end{definition}
\begin{theorem}[Stability estimate]
Let Assumptions 1 and 2 be satisfied and let $\left(\boldsymbol{u}_h,(p_{kh})_{k\in J}\right)$ be the solution of Equation \eqref{eq:summed} for any $t\in(0,\hat{t}]$. Let the stability parameters be large enough for any $k\in J$. then, it holds:
\begin{equation}
||(\boldsymbol{u}_h, (p_{kh})_{k\in J})(\hat{t})||_\varepsilon
        \lesssim \vartheta_0 + \int_0^{\hat{t}} \left(\dfrac{1}{\sqrt{\rho}}||\boldsymbol{f}(t)|| + \sum_{k \in J} \dfrac{1}{\sqrt{c_k}}||g_k(t)|| \right) \mathrm{d}t,  
\end{equation}
where we use the following definition:
\begin{equation}
\label{eq:incond}
    \vartheta_0^2 := ||\sqrt{\rho}\dot{\boldsymbol{u}}_h^0||^2 + ||\boldsymbol{u}_h^0||_\mathrm{DG,E}^2 + \sum_{k \in J} ||\sqrt{c_k} p_{kh}^0||^2 
\end{equation}
\end{theorem}
\begin{proof}
We start from the Equation \eqref{eq:summed} and we choose $\boldsymbol{v}_h=\dot{\boldsymbol{u}}_h$ and $q_{kh}=p_{kh}$. Then we find:
\begin{equation*}
    \begin{split}
        \rho\left(\ddot{\boldsymbol{u}}_h,\dot{\boldsymbol{u}}_h\right)_\Omega +  \mathcal{A}_\mathrm{E}(\boldsymbol{u}_h,\dot{\boldsymbol{u}}_h) + \sum_{k \in J}\Bigg( -&\mathcal{B}_k(p_{kh},\dot{\boldsymbol{u}}_h) + c_k\left(\dot{p}_{kh},p_{kh}\right)_{\Omega} + \mathcal{A}_{\mathrm{P}_k}(p_{kh},p_{kh}) \\ + &  \mathcal{B}_k\left(p_{kh},
    \dot{\boldsymbol{u}}_h\right) + C_k\left((p_{jh})_{j\in J},q_{kh}\right) \Bigg) = F(\dot{\boldsymbol{u}}_h) + \sum_{k \in J} G_k(p_{kh}).
    \end{split}
\end{equation*}
This choice allows us to simplify the bilinear form $\mathcal{B}_k$, because for any $k\in J$ it appears in the equation with different signs. Then, we obtain:
\begin{equation*}
        \rho\left(\ddot{\boldsymbol{u}}_h,\dot{\boldsymbol{u}}_h\right)_\Omega + \mathcal{A}_\mathrm{E}(\boldsymbol{u}_h,\dot{\boldsymbol{u}}_h) + \sum_{k \in J}\Bigg( c_k\left(\dot{p}_{kh},p_{kh}\right)_{\Omega} + \mathcal{A}_{\mathrm{P}_k}(p_{kh},p_{kh}) + C_k\left((p_{jh})_{j\in J},q_{kh}\right)\Bigg) = F(\dot{\boldsymbol{u}}_h) + \sum_{k \in J} G_k(p_{kh}).
\end{equation*}
Now, we recall the integration by parts formula:
\begin{equation}
    \int_0^t (\dot{v}(s),w(s))_* \mathrm{d}s = (v(t),w(t))_* - (v(0),w(0))_* - \int_0^t (v(s),\dot{w}(s))_* \mathrm{d}s
\end{equation}
which holds for each $v$ and $w$ regular enough and for any scalar product $(\cdot,\cdot)_*$.
\par
The application of this gives rise to the following estimates:
\begin{equation*}
    \begin{split}
        \int_0^t \rho\left(\ddot{\boldsymbol{u}}_h(s),\dot{\boldsymbol{u}}_h(s)\right)_\Omega \mathrm{d}s & = \rho\left(\dot{\boldsymbol{u}}_h(t),\dot{\boldsymbol{u}}_h(t)\right)_\Omega - \rho\left(\dot{\boldsymbol{u}}_h^0,\dot{\boldsymbol{u}}_h^0\right)_\Omega - \int_0^t \rho\left(\ddot{\boldsymbol{u}}_h(s),\dot{\boldsymbol{u}}_h(s)\right)_\Omega \mathrm{d}s \\
        \int_0^t \mathcal{A}_\mathrm{E}(\boldsymbol{u}_h(s),\dot{\boldsymbol{u}}_h(s))\mathrm{d}s & = \mathcal{A}_\mathrm{E}(\boldsymbol{u}_h(t),\boldsymbol{u}_h(t)) - \mathcal{A}_\mathrm{E}(\boldsymbol{u}_h^0,\boldsymbol{u}_h^0)  - \int_0^t\mathcal{A}_\mathrm{E}(\boldsymbol{u}_h(s),\dot{\boldsymbol{u}}_h(s))\mathrm{d}s \mathrm{d}s \\
        \int_0^t c_k\left(\dot{p}_{kh}(s),{p}_{kh}(s)\right)_\Omega \mathrm{d}s & = c_k\left(p_{kh}(t),p_{kh}(t)\right)_\Omega - c_k\left(p_{kh}^0,p_{kh}^0\right)_\Omega - \int_0^t c_k\left(\dot{p}_{kh}(s),{p}_{kh}(s)\right)_\Omega \mathrm{d}s
    \end{split}
\end{equation*}
Then, integrating the equation, we obtain:
\begin{equation*}
    \begin{split}
        ||\sqrt{\rho}&\dot{\boldsymbol{u}}_h(t)||^2 -  ||\sqrt{\rho}\dot{\boldsymbol{u}}_h^0||^2 + \mathcal{A}_\mathrm{E}(\boldsymbol{u}_h(t),\boldsymbol{u}_h(t)) - \mathcal{A}_\mathrm{E}(\boldsymbol{u}_h^0,\boldsymbol{u}_h^0) + \sum_{k \in J}\Bigg( ||\sqrt{c_k} p_{kh}(t)||^2 - ||\sqrt{c_k} p_{kh}^0||^2  + \\ + & 2\int_0^t\mathcal{A}_{\mathrm{P}_k}(p_{kh}(s),p_{kh}(s)) \mathrm{d}s + 
        2\int_0^tC_k\left((p_{jh}(s))_{j\in J},q_{kh}(s)\right) \mathrm{d}s \Bigg) = 2\int_0^t F(\dot{\boldsymbol{u}}_h(s)) \mathrm{d}s + 2\sum_{k \in J} \int_0^t G_k(p_{kh}(s)) \mathrm{d}s.
    \end{split}
\end{equation*}
Now we can use continuity and coercivity estimates we stated in Equation \eqref{eq:incond}:
\begin{equation*}
    \begin{split}
        ||(&\boldsymbol{u}_h,(p_{kh})_{k\in J})(t)||_\varepsilon^2 \leq \\
        \leq & ||\sqrt{\rho}\dot{\boldsymbol{u}}_h(t)||^2 +  ||\boldsymbol{u}_h(t)||_\mathrm{DG,E}^2  + \sum_{k \in J}\left( ||\sqrt{c_k} p_{kh}(t)||^2 + 2\int_0^t \left(||(p_{kh}(s)||_{\mathrm{DG,P}_k}^2 + ||\sqrt{\beta_k^\mathrm{e}}\;p_{kh}(s)||^2\right) \mathrm{d}s \right) \lesssim  \\
        \lesssim &  ||\sqrt{\rho}\dot{\boldsymbol{u}}_h^0||^2 + ||\boldsymbol{u}_h^0||_\mathrm{DG,E}^2 + \sum_{k \in J} ||\sqrt{c_k} p_{kh}^0||^2 + 2\int_0^t F(\dot{\boldsymbol{u}}_h(s)) \mathrm{d}s + 2\sum_{k \in J} \int_0^t G_k(p_{kh}(s)) \mathrm{d}s = \\
        = & \vartheta_0^2 + 2\int_0^t F(\dot{\boldsymbol{u}}_h(s)) \mathrm{d}s + 2\sum_{k \in J} \int_0^t G_k(p_{kh}(s)) \mathrm{d}s
    \end{split}
\end{equation*}
Then we use Equation \eqref{eq:incond} and then the continuity of the linear functionals, to obtain:
\begin{equation*}
    \begin{split}
        ||(\boldsymbol{u}_h, (p_{kh})_{k\in J})(t)||_\varepsilon^2 \lesssim & \vartheta_0^2 + 2\int_0^t ||\boldsymbol{f}(s)||\;||\dot{\boldsymbol{u}}_h(s)|| \mathrm{d}s + 2\sum_{k \in J} \int_0^t ||g_k(s)||\;||p_{kh}(s)|| \mathrm{d}s \lesssim \\
        \lesssim &   \vartheta_0^2 + \int_0^t \dfrac{2}{\sqrt{\rho}}||\boldsymbol{f}(s)||\;||\sqrt{\rho}\dot{\boldsymbol{u}}_h(s)|| \mathrm{d}s + \sum_{k \in J} \int_0^t \dfrac{2}{\sqrt{c_k}}||g_k(s)||\;||\sqrt{c_k}p_{kh}(s)|| \mathrm{d}s \lesssim \\
        \lesssim & \vartheta_0^2 + \int_0^t \left(\dfrac{2}{\sqrt{\rho}}||\boldsymbol{f}(s)|| + \sum_{k \in J} \dfrac{2}{\sqrt{c_k}}||g_k(s)|| \right)||(\boldsymbol{u}_h,(p_{kh})_{k\in J})(s)||_\varepsilon \mathrm{d}s
    \end{split}
\end{equation*}
Using Gr\"{o}nwall Lemma \cite{quarteroni:EDP}, we reach the thesis:
\begin{equation*}
    \begin{split}
        ||(\boldsymbol{u}_h, (p_{kh})_{k\in J})(t)||_\varepsilon
        \lesssim \vartheta_0 + \int_0^t \left(\dfrac{1}{\sqrt{\rho}}||\boldsymbol{f}(s)|| + \sum_{k \in J} \dfrac{1}{\sqrt{c_k}}||g_k(s)|| \right) \mathrm{d}s
    \end{split}
\end{equation*}
\end{proof}

\section{Error analysis}
\label{sec:error}
In this section, we derive an a priori error estimate for the solution of the PolyDG semi-discrete problem \eqref{eq:summed}. For the sake of simplicity we neglect the dependencies of the inequality constants on the model parameters, using the notation $x\lesssim y$ to say that $\exists C>0: x\leq C y$, where $C$ is function of the model parameters (but it is independent of the discretization parameters).
\par
First of all, we need to introduce the following definition:
\begin{equation}
    |||p|||_{\mathrm{DG,P}_j} = ||p||_{\mathrm{DG,P}_j}
    + \Big|\Big|\zeta_j^{-\frac{1}{2}}\dfrac{1}{\mu_j}\averagel \mathbf{K}_j\nabla_h p\averager\Big|\Big|_{\mathrm{L}^2(\facesinternal\cup\facesDj)} \qquad\forall p \in H^2(\partition),
\end{equation}
\begin{equation}
    |||\boldsymbol{v}|||_\mathrm{DG,E} =  ||\boldsymbol{v}||_\mathrm{DG,E} + \Big|\Big|\eta^{-\frac{1}{2}}\averagel\sqrt{\mathbb{C}_\mathrm{E}}[\boldsymbol{\varepsilon}_h(\boldsymbol{v})]\averager \Big|\Big|_{\mathrm{L}^2(\facesinternal\cup\facesD)} \qquad\forall \boldsymbol{v} \in \mathbf{H}^2(\partition,\mathbb{R}^d),
\end{equation}
\par
We introduce the interpolants of the solutions $\boldsymbol{u}_\mathrm{I}\in\Vh$ and $p_{k\mathrm{I}}\in\Qh$ of the continuous formulation \eqref{eq:weakform}. Then, for a polytopic mesh $\partition$ which satisfies Assumption 2, we can define a Stein operator $\mathcal{E}:H^m(K)\rightarrow H^m(\mathbb{R}^d)$ for any $K\in\partition$ and $m\in\mathbb{N}_0$ such that:
\begin{equation*}
    \mathcal{E}v|_K = v \qquad ||\mathcal{E}v||_{H^m(\mathbb{R}^d)} \lesssim ||v||_{H^m(K)}, \qquad \forall v\in H^m(K).
\end{equation*}
\begin{proposition}
Let Assumption 2 be fulfilled. If $d\geq 2$, then the following estimates hold:
\begin{equation}
    \forall \boldsymbol{v}\in H^n(\partition;\mathbb{R}^d)\quad \exists \boldsymbol{v}_\mathrm{I}\in\Vh:\quad|||\boldsymbol{v}-\boldsymbol{v}_\mathrm{I}|||^2_\mathrm{DG,E} \lesssim \sum_{K\in\partition} \dfrac{h_K^{2\min\{p+1,n\}-2}}{p^{2n-3}} ||\mathcal{E}\boldsymbol{v}||^2_{\mathbf{H}^n(K,\mathbb{R}^d)},
\end{equation}
\begin{equation}
    \forall p_j\in\mathbf{H}^n(\partition)\quad \exists p_{j\mathrm{I}}\in\Qh:\quad|||p_j-p_{j\mathrm{I}}|||^2_{\mathrm{DG,P}_j} \lesssim \sum_{K\in\partition} \dfrac{h_K^{2\min\{q+1,n\}-2}}{q^{2n-3}} ||\mathcal{E}p_j||^2_{H^n(K)}.
\end{equation}
\end{proposition}

\subsection{Error estimates}
First of all let us consider $(\boldsymbol{u}_h,(p_{kh})_{k\in J})$ solution of \eqref{eq:DGFormulation} and $(\boldsymbol{u},(p_{k})_{k\in J})$ solution of \eqref{eq:weakform}. To extend the bilinear forms of \eqref{eq:DGFormulation} to the space of continuous solutions we need further regularity requirements. We assume element-wise $H^2$-regularity of the displacement and pressures together with the continuity of the normal stress and fluid flow across the interfaces $F\in\facesinternal$ for all time $t\in(0,T]$. In this context, we need to provide additional boundedness results for the functionals of the formulation:
\begin{proposition}
Let Assumption 2 be satisfied
Then:
\begin{equation}
    |\mathcal{A}_\mathrm{E}(\boldsymbol{v},\boldsymbol{w}_h)| \lesssim |||\boldsymbol{v}|||_\mathrm{DG,E} ||\boldsymbol{w}_h||_\mathrm{DG,E}, \qquad \forall \boldsymbol{v}\in  \mathbf{H}^2(\mathcal{T}_h;\mathbb{R}^d),\forall\boldsymbol{w}_h\in \mathbf{V}_h^\mathrm{DG}
\end{equation}
\begin{equation}
    |\mathcal{A}_{\mathrm{P}_j}(p_{j},q_{jh})| \lesssim |||p_{j}|||_{\mathrm{DG,P}_j}||q_{jh}||_{\mathrm{DG,P}_j}, \qquad \forall p_{j}\in H^2(\partition),\forall q_{jh}\in Q_h^\mathrm{DG}
\end{equation}
\begin{equation}
    |\mathcal{B}_k(q_{kh},\boldsymbol{v})|\lesssim|||\boldsymbol{v}|||_{\mathrm{DG,E}}||q_{kh}|| \qquad \forall \boldsymbol{v}\in  \mathbf{H}^2(\mathcal{T}_h;\mathbb{R}^d),\; \forall q_{kh}\in\Qh
\end{equation}
\begin{equation}
    |\mathcal{B}_k(q_{k},\boldsymbol{v}_h)|\lesssim||\boldsymbol{v}||_{\mathrm{DG,E}}|||q_{kh}|||_{\mathrm{DG,P}_j} \qquad \forall \boldsymbol{v}_h\in\Vh,\; \forall q_{k}\in H^2(\partition)
\end{equation}
\end{proposition}
The proof of these relations could be found in \cite{antonietti_high-order_2022,antonietti_high-order_2020, Bonetti:Porothermo}.
\begin{theorem}
Let Assumptions 1 and 2 be fulfilled and let $\left(\boldsymbol{u},(p_{j})_{j\in J}\right)$ be the solution of \eqref{eq:weakform} for any $t\in(0,T]$ and let it satisfy the following additional regularity requirements:
\begin{equation}
    \boldsymbol{u}\in C^1((0,T];\mathbf{H}^m(\Omega;\mathbb{R}^d)) \qquad p_j\in C^1((0,T]; H^n(\Omega)) \quad\forall j\in J
\end{equation}
for $m,n\geq 2$. Let $\left(\boldsymbol{u}_h,(p_{jh})_{j\in J}\right)$ be the solution of \eqref{eq:summed} for any $t\in(0,T]$. Then, the following estimate holds:
\begin{equation}
    \begin{split}
    |||\left(\boldsymbol{e}^u,(e^{p_j})_{j\in J}\right)(t)|||_\varepsilon^2 \lesssim &
   \sum_{K\in\partition} \dfrac{h_K^{2\min\{p+1,m\}-2}}{p^{2m-3}} \left[||\mathcal{E}\boldsymbol{u}(t)||^2_{\mathbf{H}^m(K,\mathbb{R}^d)} +\int_0^t ||\mathcal{E}\dot{\boldsymbol{u}}(s)||^2_{\mathbf{H}^m(K,\mathbb{R}^d)}\mathrm{d}s +\int_0^t ||\mathcal{E}\ddot{\boldsymbol{u}}(s)||^2_{\mathbf{H}^m(K,\mathbb{R}^d)}\mathrm{d}s\right]
   \\ 
   + & \sum_{K\in\partition} \dfrac{h_K^{2\min\{q+1,n\}-2}}{q^{2n-3}} \sum_{j \in J}\left[||\mathcal{E}p_j(t)||^2_{H^n(K)}+\int_0^t ||\mathcal{E}p_j(s)||^2_{H^n(K)} \mathrm{d}s+ \int_0^t ||\mathcal{E}\dot{p}_j(s)||^2_{H^n(K)} \mathrm{d}s \right],
    \end{split}
\end{equation}
where $\boldsymbol{e}^u=\boldsymbol{u}-\boldsymbol{u}_h$ and $e^{p_j}=p_j-p{jh}$ for any $j\in J$.
\end{theorem}
\begin{proof}
 Subtracting the resulting equation from problem \eqref{eq:DGFormulation}, we obtain:
\begin{equation*}
    \begin{split}
        \rho\left(\ddot{\boldsymbol{u}}-\ddot{\boldsymbol{u}}_h,\boldsymbol{v}_h\right)_\Omega + & \mathcal{A}_\mathrm{E}(\boldsymbol{u}-\boldsymbol{u}_h,\boldsymbol{v}_h) + \sum_{k \in J}\Bigg( -\mathcal{B}_k(p_k-p_{kh},\boldsymbol{v}_h) + c_k\left(\dot{p}_k-\dot{p}_{kh},q_{kh}\right)_{\Omega}  \\ + &  \mathcal{A}_{\mathrm{P}_k}(p_k-p_{kh},q_{kh}) + \mathcal{B}_k\left(q_{kh},
    \dot{\boldsymbol{u}}-\dot{\boldsymbol{u}}_h\right) + C_k\left((p_j-p_{jh})_{j\in J},q_{kh}\right)\Bigg) = 0.
    \end{split}
\end{equation*}
We define the errors for the displacement $\boldsymbol{e}^u_h = \boldsymbol{u}_\mathrm{I}-\boldsymbol{u}_h$ and $\boldsymbol{e}^u_\mathrm{I} = \boldsymbol{u}-\boldsymbol{u}_\mathrm{I}$. Analogously, for the pressures $e^{p_k}_h = p_{k\mathrm{I}}-p_{kh}$ and $e^{p_k}_\mathrm{I} = p_k - p_{k\mathrm{I}}$. Then we can rewrite the equation above as follows:
\begin{equation*}
    \begin{split}
        \rho\left(\ddot{\boldsymbol{e}}^u_h,\dot{\boldsymbol{e}}^u_h\right)_\Omega + & \mathcal{A}_\mathrm{E}(\boldsymbol{e}^u_h,\dot{\boldsymbol{e}}^u_h) + \sum_{k \in J}\Bigg( -\mathcal{B}_k(e^{p_k}_h,\dot{\boldsymbol{e}}^u_h) + c_k\left(\dot{e}^{p_k}_h,e^{p_k}_h\right)_{\Omega} + \mathcal{A}_{\mathrm{P}_k}(e^{p_k}_h,e^{p_k}_h) + \mathcal{B}_k\left(e^{p_k}_h,\dot{\boldsymbol{e}}^u_h\right)  \\ + &  C_k\left((e^{p_j}_h)_{j\in J},e^{p_k}_h\right)\Bigg) = \rho\left(\ddot{\boldsymbol{e}}^u_\mathrm{I},\dot{\boldsymbol{e}}^u_h\right)_\Omega + \mathcal{A}_\mathrm{E}(\boldsymbol{e}^u_\mathrm{I},\dot{\boldsymbol{e}}^u_h) + \sum_{k \in J}\Bigg( -\mathcal{B}_k(e^{p_k}_\mathrm{I},\dot{\boldsymbol{e}}^u_h) + c_k\left(\dot{e}^{p_k}_\mathrm{I},e^{p_k}_h\right)_{\Omega}  \\ + & \mathcal{A}_{\mathrm{P}_k}(e^{p_k}_\mathrm{I},e^{p_k}_h) + \mathcal{B}_k\left(e^{p_k}_h,\dot{\boldsymbol{e}}^u_\mathrm{I}\right) + C_k\left((e^{p_j}_\mathrm{I})_{j\in J},e^{p_k}_h\right) \Bigg)
    \end{split}
\end{equation*}
Due to the symmetry of scalar product and $\mathcal{A}_\mathrm{E}$ we can rewrite the problem:
\begin{equation*}
    \begin{split}
        \dfrac{\rho}{2}\dfrac{\mathrm{d}}{\mathrm{d}t}\left(\dot{\boldsymbol{e}}^u_h,\dot{\boldsymbol{e}}^u_h\right)_\Omega + & \dfrac{1}{2}\dfrac{\mathrm{d}}{\mathrm{d}t}\mathcal{A}_\mathrm{E}(\boldsymbol{e}^u_h,\boldsymbol{e}^u_h) + \sum_{k \in J}\Bigg( \dfrac{c_k}{2}\dfrac{\mathrm{d}}{\mathrm{d}t}\left(e^{p_k}_h,e^{p_k}_h\right)_{\Omega} + \mathcal{A}_{\mathrm{P}_k}(e^{p_k}_h,e^{p_k}_h) + C_k\left((e^{p_j}_h)_{j\in J},e^{p_k}_h\right)\Bigg)  \\ = & \rho\left(\ddot{\boldsymbol{e}}^u_\mathrm{I},\dot{\boldsymbol{e}}^u_h\right)_\Omega + \dfrac{\mathrm{d}}{\mathrm{d}t}\mathcal{A}_\mathrm{E}(\boldsymbol{e}^u_\mathrm{I},\boldsymbol{e}^u_h) - \mathcal{A}_\mathrm{E}(\dot{\boldsymbol{e}^u_\mathrm{I}},\boldsymbol{e}^u_h) + \sum_{k \in J}\Bigg( -\dfrac{\mathrm{d}}{\mathrm{d}t}\mathcal{B}_k(e^{p_k}_\mathrm{I},\boldsymbol{e}^u_h) + \mathcal{B}_k(\dot{e^{p_k}_\mathrm{I}},\boldsymbol{e}^u_h) + c_k\left(\dot{e}^{p_k}_\mathrm{I},e^{p_k}_h\right)_{\Omega}  \\ + & \mathcal{B}_k(e^{p_k}_h,\dot{\boldsymbol{e}}^u_\mathrm{I}) + \mathcal{A}_{\mathrm{P}_k}(e^{p_k}_\mathrm{I},e^{p_k}_h) + C_k\left((e^{p_j}_\mathrm{I})_{j\in J},e^{p_k}_h\right)\Bigg).
    \end{split}
\end{equation*}
Now we integrate between $0$ and $t$. We remark that $\boldsymbol{e}^u_h(0)=\boldsymbol{0}$,$\dot{\boldsymbol{e}}^u_h(0)=\boldsymbol{0}$ and $e^{p_k}_h=0$ for each $k\in J$. Then, by proceeding in an analogous way to what we did in the proof of Theorem 1, we obtain:
\begin{equation*}
    \begin{split}
    |||\left(\boldsymbol{e}^u_h,(e^{p_k}_h)_{k\in J}\right)(t)|||_\varepsilon^2 \lesssim & 
    \mathcal{A}_\mathrm{E}(\boldsymbol{e}^u_\mathrm{I}(t),\boldsymbol{e}^u_h(t)) - \sum_{k \in J}\mathcal{B}_k(e^{p_k}_\mathrm{I}(t),\boldsymbol{e}^u_h(t)) +
    \int_0^t\rho\left(\ddot{\boldsymbol{e}}^u_\mathrm{I}(s),\dot{\boldsymbol{e}}^u_h(s)\right)_\Omega - \int_0^t\mathcal{A}_\mathrm{E}(\dot{\boldsymbol{e}}^u_\mathrm{I}(s),\boldsymbol{e}^u_h(s)) \\ + & \sum_{k \in J}\Bigg(\int_0^t\mathcal{B}_k(\dot{e^{p_k}_\mathrm{I}}(s),\boldsymbol{e}^u_h(s)) + \int_0^t c_k\left(\dot{e}^{p_k}_\mathrm{I}(s),e^{p_k}_h(s)\right)_{\Omega} + \int_0^t\mathcal{B}_k(e^{p_k}_h(s),\dot{\boldsymbol{e}}^u_\mathrm{I}(s)) \\ + & \int_0^t\mathcal{A}_{\mathrm{P}_k}(e^{p_k}_\mathrm{I}(s),e^{p_k}_h(s)) + \int_0^t C_k\left((e^{p_j}_\mathrm{I}(s))_{j\in J},e^{p_k}_h(s)\right)\Bigg).
    \end{split}
\end{equation*}
Then exploiting the continuity relations in Proposition 5:
\begin{equation*}
    \begin{split}
    |||\left(\boldsymbol{e}^u_h,(e^{p_k}_h)_{k\in J}\right)(t)|||_\varepsilon^2 \lesssim & 
   |||\boldsymbol{e}^u_\mathrm{I}(t)|||_\mathrm{DG,E}||\boldsymbol{e}^u_h(t)||_\mathrm{DG,E} + \sum_{k \in J}|||e^{p_k}_\mathrm{I}(t)|||_{\mathrm{DG,P}_j}||\boldsymbol{e}^u_h(t)||_{DG,E} \\ 
   + & \int_0^t||\sqrt{\rho}\ddot{\boldsymbol{e}}^u_\mathrm{I}(s)||\;||\sqrt{\rho}\dot{\boldsymbol{e}}^u_h(s)|| + \int_0^t|||\dot{\boldsymbol{e}}^u_\mathrm{I}(s)|||_\mathrm{DG,E}||\boldsymbol{e}^u_h(s)||_\mathrm{DG,E} \\ 
   + & \sum_{k \in J}\Bigg(\int_0^t|||\dot{e}^{p_k}_\mathrm{I}(s)|||_{\mathrm{DG,P}_j}||\boldsymbol{e}^u_h(s)||_{DG,E} + \int_0^t ||\sqrt{c_k}\dot{e}^{p_k}_\mathrm{I}(s)||\;||\sqrt{c_k}e^{p_k}_h(s)|| \\
   + & \int_0^t (||e^{p_k}_h(s)||_{\mathrm{DG,P}_j}|||\dot{\boldsymbol{e}}^u_\mathrm{I}(s)|||_{DG,E} 
   + \int_0^t||e^{p_k}_h(s)||_{\mathrm{DG,P}_j}|||e^{p_k}_\mathrm{I}(s)|||_{\mathrm{DG,P}_j} \\
   + & \sum_{j \in J}\sum_{k \in J}\int_0^t\mathbb{B}||c_j^{-\frac{1}{2}}e^{p_j}_\mathrm{I}(s)||\;||\sqrt{c_k}e^{p_k}_h(s)||\Bigg).
    \end{split}
\end{equation*}
Then using the definition of the energy norm and both H\"older and Young inequalities we obtain:
\begin{equation*}
    \begin{split}
    |||&\left(\boldsymbol{e}^u_h,(e^{p_k}_h)_{k\in J}\right)(t)|||_\varepsilon^2 \lesssim 
   |||\boldsymbol{e}^u_\mathrm{I}(t)|||_\mathrm{DG,E}^2 + \sum_{k \in J}|||e^{p_k}_\mathrm{I}(t)|||_{\mathrm{DG,P}_j}^2 + \int_0^t \left(|||\dot{\boldsymbol{e}}^u_\mathrm{I}(s)|||_\mathrm{DG,E} ^2
   +\sum_{k \in J}|||e^{p_k}_\mathrm{I}(s)|||_{\mathrm{DG,P}_k}^2\right) \\ 
   + & \int_0^t||\left(\boldsymbol{e}^u_h,(e^{p_k}_h)_{k\in J}\right)(s)||_\varepsilon \Bigg(||\sqrt{\rho}\ddot{\boldsymbol{e}}^u_\mathrm{I}(s)||+|||\dot{\boldsymbol{e}}^u_\mathrm{I}(s)|||_\mathrm{DG,E} + \sum_{k \in J}\Big(|||\dot{e}^{p_k}_\mathrm{I}(s)|||_{\mathrm{DG,P}_j} + ||\sqrt{c_k}\dot{e}^{p_k}_\mathrm{I}(s)||+||c_k^{-\frac{1}{2}}e^{p_k}_\mathrm{I}(s)||\Big)\Bigg).
    \end{split}
\end{equation*}
Then by application of the Gr\"onwall lemma \cite{quarteroni:EDP}, we obtain:
\begin{equation*}
    \begin{split}
    |||\left(\boldsymbol{e}^u_h,(e^{p_k}_h)_{k\in J}\right)(t)&|||_\varepsilon^2 \lesssim 
   |||\boldsymbol{e}^u_\mathrm{I}(t)|||_\mathrm{DG,E}^2 + \sum_{k \in J}|||e^{p_k}_\mathrm{I}(t)|||_{\mathrm{DG,P}_j}^2 + \int_0^t \left(|||\dot{\boldsymbol{e}}^u_\mathrm{I}(s)|||_\mathrm{DG,E}^2 
   +\sum_{k \in J}|||e^{p_k}_\mathrm{I}(s)|||_{\mathrm{DG,P}_k}^2\right) \\ 
   + & \int_0^t \left(||\sqrt{\rho}\ddot{\boldsymbol{e}}^u_\mathrm{I}(s)||^2+|||\dot{\boldsymbol{e}}^u_\mathrm{I}(s)|||_\mathrm{DG,E}^2 + \sum_{k \in J}\Big(|||\dot{e}^{p_k}_\mathrm{I}(s)|||_{\mathrm{DG,P}_j}^2 + ||\sqrt{c_k}\dot{e}^{p_k}_\mathrm{I}(s)||^2+||c_k^{-\frac{1}{2}}e^{p_k}_\mathrm{I}(s)||^2\Big)\right).
    \end{split}
\end{equation*}
Then by using the relations of Proposition 4, we find:
\begin{equation*}
\label{eq:resproof}
    \begin{split}
    |||\left(\boldsymbol{e}^u_h,(e^{p_j}_h)_{j\in J}\right)(t)|||_\varepsilon^2 \lesssim &
   \sum_{K\in\partition} \dfrac{h_K^{2\min\{p+1,n\}-2}}{p^{2n-3}} \left[||\mathcal{E}\boldsymbol{u}(t)||^2_{\mathbf{H}^n(K,\mathbb{R}^d)} +\int_0^t ||\mathcal{E}\dot{\boldsymbol{u}}(s)||^2_{\mathbf{H}^n(K,\mathbb{R}^d)}\mathrm{d}s +\int_0^t ||\mathcal{E}\ddot{\boldsymbol{u}}(s)||^2_{\mathbf{H}^n(K,\mathbb{R}^d)}\mathrm{d}s\right]
   \\ 
   + & \sum_{K\in\partition} \dfrac{h_K^{2\min\{q+1,n\}-2}}{q^{2n-3}} \sum_{j \in J}\left[||\mathcal{E}p_j(t)||^2_{H^n(K)}+\int_0^t ||\mathcal{E}p_j(s)||^2_{H^n(K)} \mathrm{d}s+ \int_0^t ||\mathcal{E}\dot{p}_j(s)||^2_{H^n(K)} \mathrm{d}s \right].
    \end{split}
\end{equation*}
Then, we use the triangular inequality to estimate the discretization error.
\begin{equation*}
    |||\left(\boldsymbol{e}^u,(e^{p_j})_{j\in J}\right)(t)|||_\varepsilon^2 \leq  |||\left(\boldsymbol{e}^u_h,(e^{p_j}_h)_{j\in J}\right)(t)|||_\varepsilon^2 + |||\left(\boldsymbol{e}^u_\mathrm{I},(e^{p_j}_\mathrm{I})_{j\in J}\right)(t)|||_\varepsilon^2
\end{equation*}
Finally, by applying the result in Equation \eqref{eq:resproof} and the interpolation error, the thesis follows.
\end{proof}
\section{Time discretization}
\label{sec:temporaldisc}
By fixing a basis for the discrete spaces $(\boldsymbol{\varphi}_n)_{n=0}^{N_u}\subset \Vh$ and $(\psi_n)_{n=0}^{N_p}\subset \Qh$, where $N_u = \dim(\Vh)$ and $N_p = \dim(\Qh)$, such that:
\begin{equation}
    \boldsymbol{u}_h(t) = \sum_{n=0}^{N_u} U_n(t)\boldsymbol{\varphi}_n\qquad p_{kh}(t) = \sum_{n=0}^{N_p} P_{kn}(t)\psi_n\quad \forall k \in J.
\end{equation}
We connect the coefficients of the expansion of $\boldsymbol{u}_h$ and $p_{jh}$ for any $j\in J$ in such a basis in the vectors $\boldsymbol{U}\in\mathbb{R}^{3N_u}$ and $\boldsymbol{P}_k\in\mathbb{R}^{N_p}$ for any $k\in J$. By using the same basis, we are able to define the following matrices:
\begin{equation*}
[\mathrm{M}_{\boldsymbol{u}}]_{ij} = (\rho \boldsymbol{\varphi}_j, \boldsymbol{\varphi}_i)_\Omega\quad\mathrm{(Elasticity\;mass\;matrix)} \qquad 
[\mathrm{K}_{\boldsymbol{u}}]_{ij} = \mathcal{A}_\mathrm{E}(\boldsymbol{\varphi}_j, \boldsymbol{\varphi}_i)\quad\mathrm{(Elasticity\;stiffness\;matrix)}
\end{equation*}
\smallskip
\begin{equation*}
[\mathrm{M}_{k}]_{ij} = (c_k \psi_j, \psi_i)_\Omega\quad(k\mathrm{-th\;pressure\;mass\;matrix)} \qquad
[\mathrm{K}_k]_{ij} = \mathcal{A}_{\mathrm{P}_k}(\psi_j, \psi_i)\quad(k\mathrm{-th\;pressure\;stiffness\;matrix)}
\end{equation*}
\smallskip
\begin{equation*}
[\mathrm{B}_k]_{ij} = \mathcal{B}_k(\psi_j, \boldsymbol{\varphi}_i)\quad\mathrm{(Pressure-displacement\;coupling\;matrix)}
\end{equation*}
\smallskip
\begin{equation*}
[\mathrm{C}_{kl}]_{ij} = (\beta_{kl}\psi_j, \psi_i)_\Omega\quad\mathrm{(Pressure-pressure\;coupling\;matrix)}
\end{equation*}
\smallskip
\begin{equation*}
[\mathrm{C}_{k}^\mathrm{e}]_{ij} = (\beta_{k}^\mathrm{e} \psi_j, \psi_i)_\Omega\quad\mathrm{(Pressure\;external-coupling\;matrix)}
\end{equation*}
Moreover, we define the forcing terms:
\begin{equation*}
[\boldsymbol{F}]_{j} = F(\boldsymbol{\varphi}_j) \qquad [\boldsymbol{G}_k]_{j} = G_k(\psi_j)
\end{equation*}
By exploiting all these definitions, we rewrite the problem \eqref{eq:DGFormulation} in algebraic form:
\begin{equation}
\label{eq:algfull}
    \begin{dcases}
         \mathrm{M}_{\boldsymbol{u}}\ddot{\boldsymbol{U}}(t)+\mathrm{K}_{\boldsymbol{u}}\boldsymbol{U}(t)-\sum_{k\in J} \mathrm{B}_k^\top \boldsymbol{P}_k(t) = \boldsymbol{F}(t), & t\in(0,T) \\[4pt]
         \mathrm{M}_k \dot{\boldsymbol{P}}_k(t)+\mathrm{B}_k \dot{\boldsymbol{U}}(t) + \mathrm{K}_k \boldsymbol{P}_k(t) + \sum_{j\in J} \mathrm{C}_{kj}(\boldsymbol{P}_k(t)-\boldsymbol{P}_j(t)) + \mathrm{C}_k^\mathrm{e} \boldsymbol{P}_k(t) = \boldsymbol{G}_k(t), & t\in(0,T)\quad \forall k \in J \\[4pt]
         \boldsymbol{U}(0) = \boldsymbol{U}_0 \\[8pt]
         \dot{\boldsymbol{U}}(0) = \boldsymbol{V}_0 \\[8pt]
         \boldsymbol{P}_k(0) = \boldsymbol{P}_{k0}, & \forall k \in J  
    \end{dcases}
\end{equation}
Then after the introduction of the vector variable $\boldsymbol{P} = [\boldsymbol{P}_{j_1}, \boldsymbol{P}_{j_2}, ..., \boldsymbol{P}_{j_n}]^\top$ with $j_1,j_2,...,j_n\in J$, we can construct the following matrices:
\begin{equation*}
\mathrm{B} = 
\begin{bmatrix}
\mathrm{B}_{j_1} \\
\mathrm{B}_{j_2} \\
\vdots \\
\mathrm{B}_{j_n}
\end{bmatrix},
\qquad
\mathrm{M}_p = 
\begin{bmatrix}
\mathrm{M}_{j_1} & 0 & \cdots & 0 \\
0 & \mathrm{M}_{j_2} & \cdots & 0 \\
\vdots & \vdots & \ddots & \vdots \\
0 & 0 & \cdots & \mathrm{M}_{j_n}
\end{bmatrix},
\qquad
\boldsymbol{G} = 
\begin{bmatrix}
\boldsymbol{G}_{j_1} \\
\boldsymbol{G}_{j_2} \\
\vdots \\
\boldsymbol{G}_{j_n}
\end{bmatrix},
\end{equation*}
\begin{equation*}
\mathrm{K}_p = 
\begin{bmatrix}
\mathrm{K}_{j_1}+\displaystyle\sum_{i\in J} \mathrm{C}_{j_1 i}  + \mathrm{C}^\mathrm{e}_{j_1} & -\mathrm{C}_{j_1 j_2} & \cdots & -\mathrm{C}_{j_1 j_n} \\
-\mathrm{C}_{j_2 j_1} & \mathrm{K}_{j_2}+\displaystyle\sum_{i\in J} \mathrm{C}_{j_2 i} + \mathrm{C}^\mathrm{e}_{j_2}& \cdots & -\mathrm{C}_{j_2 j_n} \\
\vdots & \vdots & \ddots & \vdots \\
-\mathrm{C}_{j_n j_1} & -\mathrm{C}_{j_n j_2} & \cdots & \mathrm{K}_{j_n}+\displaystyle\sum_{i\in J} \mathrm{C}_{j_n i}  + \mathrm{C}^\mathrm{e}_{j_n}
\end{bmatrix},
\end{equation*}
Then, we write the Equation \eqref{eq:algfull} in a compact form, as follows:
\begin{equation}
    \begin{dcases}
         \mathrm{M}_{\boldsymbol{u}}\ddot{\boldsymbol{U}}(t)+\mathrm{K}_{\boldsymbol{u}}\boldsymbol{U}(t)-\mathrm{B}^\top \boldsymbol{P}(t) = \boldsymbol{F}(t), & t\in(0,T) \\[8pt]
         \mathrm{M}_p \dot{\boldsymbol{P}}(t)+\mathrm{B} \dot{\boldsymbol{U}}(t) + \mathrm{K}_p \boldsymbol{P}(t) = \boldsymbol{G}(t), & t\in(0,T) \\[8pt]
         \boldsymbol{U}(0) = \boldsymbol{U}_0 \qquad
         \dot{\boldsymbol{U}}(0) = \boldsymbol{V}_0 \\[8pt]
         \boldsymbol{P}(0) = \boldsymbol{P}_{0}
    \end{dcases}
\end{equation}
\par
Let now construct a temporal discretization of the interval $(0,T)$ by constructing a partition of $N$ intervals $0=t_0<t_1<...<t_N=T$. We assume a constant timestep $\Delta t = t_{n+1}-t_n$ for each $n=0,...,N$. Let now construct a discretized formulation by means of the Newmark $\beta$-method for the first equation. We introduce a velocity vector $\boldsymbol{Z}^{n}$, and an acceleration one $\boldsymbol{A}^{n}$. 
Then, we have the following equations to be solved at each timestep $t_n$:
\begin{equation}
\begin{dcases}
\left(\dfrac{1}{\beta \Delta t^2}\mathrm{M}_{\boldsymbol{u}}+\mathrm{K}_{\boldsymbol{u}}\right)\boldsymbol{U}^{n+1} -\mathrm{B}^\top \boldsymbol{P}^{n+1} = \boldsymbol{F}^{n+1} +\dfrac{1}{\beta \Delta t^2}\mathrm{M}_{\boldsymbol{u}}\boldsymbol{U}^{n} + \dfrac{1}{\beta\Delta t} \mathrm{M}_{\boldsymbol{u}}\boldsymbol{Z}^{n} + \dfrac{1-2\beta}{2\beta}\mathrm{M}_{\boldsymbol{u}}\boldsymbol{A}^{n} \\
\boldsymbol{A}^{n+1} =  \dfrac{1}{\beta \Delta t^2}(\boldsymbol{U}^{n+1}-\boldsymbol{U}^{n}) -\dfrac{1}{\beta \Delta t}\boldsymbol{Z}^{n} + \dfrac{2\beta-1}{2\beta}\boldsymbol{A}^{n} \\
\boldsymbol{Z}^{n+1} = \boldsymbol{Z}^{n} + \Delta t(\gamma \boldsymbol{A}^{n+1} + (1-\gamma)\boldsymbol{A}^{n})
\end{dcases}
\end{equation}
We couple the problem above with a $\theta$-method for the pressure equations. To obtain the formulation we consider first the definition of the velocity $\boldsymbol{Z}=\dot{\boldsymbol{U}}$ at time-continuous level, which gives us the following expression:
\begin{equation}
    \mathrm{M}_p \dot{\boldsymbol{P}}(t)+\mathrm{B} \boldsymbol{Z}(t) + \mathrm{K}_p \boldsymbol{P}(t) = \boldsymbol{G}(t), \quad t\in(0,T).
\end{equation}
Using this form of the equation, the derivation of the discretized equation naturally follows:
\begin{equation}
\begin{split}
    \mathrm{M}_p\boldsymbol{P}^{n+1} = & \mathrm{M}_p\boldsymbol{P}^{n} + \Delta t\theta( \boldsymbol{G}^{n+1} - \mathrm{B} \boldsymbol{Z}^{n+1} - \mathrm{K}_p \boldsymbol{P}^{n+1}) + \Delta t(1-\theta)( \boldsymbol{G}^{n} - \mathrm{B} \boldsymbol{Z}^{n} - \mathrm{K}_p \boldsymbol{P}^{n}) = \\
    = & \mathrm{M}_p\boldsymbol{P}^{n} + \Delta t\theta( \boldsymbol{G}^{n+1} - \tfrac{\gamma}{\beta \Delta t}\mathrm{B}(\boldsymbol{U}^{n+1}-\boldsymbol{U}^{n}) - (1-\tfrac{\gamma}{\beta})\mathrm{B}\boldsymbol{Z}^{n} - (1-\tfrac{\gamma}{2\beta})\Delta t\mathrm{B}\boldsymbol{A}^{n}- \mathrm{K}_p \boldsymbol{P}^{n+1}) \\ + & \Delta t(1-\theta)( \boldsymbol{G}^{n} - \mathrm{B} \boldsymbol{Z}^{n} - \mathrm{K}_p \boldsymbol{P}^{n})
\end{split}
\end{equation}
The final algebraic discretized formulation reads as follows:
\begin{equation}
\label{eq:fullyalgebraicproblem}
\begin{dcases}
\left(\dfrac{1}{\beta \Delta t^2}\mathrm{M}_{\boldsymbol{u}}+\mathrm{K}_{\boldsymbol{u}}\right)\boldsymbol{U}^{n+1} -\mathrm{B}^\top \boldsymbol{P}^{n+1} = \boldsymbol{F}^{n+1} +\dfrac{1}{\beta \Delta t^2}\mathrm{M}_{\boldsymbol{u}}\boldsymbol{U}^{n} + \dfrac{1}{\beta\Delta t} \mathrm{M}_{\boldsymbol{u}}\boldsymbol{Z}^{n} + \dfrac{1-2\beta}{2\beta}\mathrm{M}_{\boldsymbol{u}}\boldsymbol{A}^{n} \\
\begin{aligned}
\left(\dfrac{1}{\Delta t}\mathrm{M}_p + \theta\mathrm{K}_p\right)\boldsymbol{P}^{n+1} + \dfrac{\theta\gamma}{\beta \Delta t}\mathrm{B}\boldsymbol{U}^{n+1} = & \theta \boldsymbol{G}^{n+1} + (1-\theta) \boldsymbol{G}^{n} + \left(\dfrac{1}{\Delta t}\mathrm{M}_p-(1-\theta)\mathrm{K}_p \right)\boldsymbol{P}^{n} + \dfrac{\theta\gamma}{\beta \Delta t}\mathrm{B}\boldsymbol{U}^{n} \\ + & \left(\dfrac{\theta\gamma}{\beta}-1\right)\mathrm{B}\boldsymbol{Z}^{n} - \theta\left(1-\dfrac{\gamma}{2\beta}\right)\Delta t\mathrm{B}\boldsymbol{A}^{n}
\end{aligned}
\\
\boldsymbol{A}^{n+1} =  \dfrac{1}{\beta \Delta t^2}(\boldsymbol{U}^{n+1}-\boldsymbol{U}^{n}) -\dfrac{1}{\beta \Delta t}\boldsymbol{Z}^{n} + \dfrac{2\beta-1}{2\beta}\boldsymbol{A}^{n} \\
\boldsymbol{Z}^{n+1} = \boldsymbol{Z}^{n} + \Delta t(\gamma \boldsymbol{A}^{n+1} + (1-\gamma)\boldsymbol{A}^{n})
\end{dcases}
\end{equation}
In order to rewrite Equation \eqref{eq:fullyalgebraicproblem} in matrix form, we introduce the following matrices:
\begin{equation*}
\mathrm{A}_1 = \begin{bmatrix}
\dfrac{\mathrm{M}_{\boldsymbol{u}}}{\beta \Delta t^2}+\mathrm{K}_{\boldsymbol{u}} & -\mathrm{B}^\top & 0 & 0 \\
\dfrac{\theta\gamma}{\beta \Delta t}\mathrm{B} & \dfrac{\mathrm{M}_p}{\Delta t}+ \theta\mathrm{K}_p & 0 & 0 \\
0 & 0 & \mathrm{I} & -\Delta t \gamma \mathrm{I} \\
-\dfrac{\mathrm{I}}{\beta \Delta t^2} & 0 & 0 & \mathrm{I} 
\end{bmatrix}
\qquad
\boldsymbol{X}^{n} = 
\begin{bmatrix}
\boldsymbol{U}^{n} \\
\boldsymbol{P}^{n} \\
\boldsymbol{Z}^{n} \\
\boldsymbol{A}^{n} 
\end{bmatrix}
\end{equation*}
\begin{equation*}
\mathrm{A}_2 = \begin{bmatrix}
\dfrac{\mathrm{M}_{\boldsymbol{u}}}{\beta \Delta t^2} & -\mathrm{B}^\top & \dfrac{\mathrm{M}_{\boldsymbol{u}}}{\beta \Delta t} &  \dfrac{1-2\beta}{2\beta}\mathrm{M}_{\boldsymbol{u}} \\
\dfrac{\theta\gamma\mathrm{B}}{\beta \Delta t} & \dfrac{\mathrm{M}_p}{\Delta t}-\tilde{\theta}\mathrm{K}_p & \left(\dfrac{\theta\gamma}{\beta}-1\right)\mathrm{B} & \left(\theta-\dfrac{\theta\gamma}{2\beta}\right)\Delta t\mathrm{B} \\
0 & 0 & \mathrm{I} & \Delta t (1-\gamma)\mathrm{I} \\
-\dfrac{\mathrm{I}}{\beta \Delta t^2} & 0 & -\dfrac{\mathrm{I}}{\beta \Delta t} & \dfrac{2\beta-1}{2\beta}\mathrm{I} 
\end{bmatrix}
\qquad
\boldsymbol{S}^{n+1} = 
\begin{bmatrix}
\boldsymbol{F}^{n+1} \\
\theta \boldsymbol{G}^{n+1} + \tilde{\theta} \boldsymbol{G}^{n} \\
0 \\
0 
\end{bmatrix} 
\end{equation*}
Finally, we the algebraic formulation reads as follows:
\begin{equation}
    \mathrm{A}_1 \boldsymbol{X}^{n+1} = \mathrm{A}_2 \boldsymbol{X}^{n} + \boldsymbol{S}^{n+1} \qquad n>0.
\end{equation}
\section{Numerical results}
\label{sec:numericalresults}
In this section, we aim at validating the accuracy of the method practice. All simulations are carried out considering the following choice of Newmark parameters $\beta=0.25$ and $\gamma=0.5$; moreover, we choose $\theta=0.5$.
\subsection{Test case 1: convergence analysis in a 3D case}
For the numerical test in this section, we use the FEniCS finite element software \cite{FenicsCode} (version 2019). We use a cubic domain with structured tetrahedral mesh. Concerning the temporal discretization, we use a timestep $\Delta t = 10^{-5}$ and a maximum time $T=5\times 10^{-3}$. We consider the following manufactured exact solution for a case with four pressure fields:
\begin{equation}
    \begin{matrix}
    \boldsymbol{u}(x,y,z,t)=\sin(\pi t)
    \begin{bmatrix}
    -\cos(\pi x)\cos(\pi y) \\
    \sin(\pi x)\sin(\pi y)  \\
    z  
    \end{bmatrix}
    \\[8pt]
    p_1(x,y,z,t)=p_3(x,y,z,t)=\pi\sin(\pi t)(\cos(\pi y)\sin(\pi x) + \cos(\pi x)\sin(\pi y)) z \\[8pt]
    p_2(x,y,z,t)=p_4(x,y,z,t)=\pi\sin(\pi t)(\cos(\pi y)\sin(\pi x) - \cos(\pi x)\sin(\pi y)) z
    \end{matrix}
\end{equation}
A fundamental assumption in this section is the isotropic permeability of the pressure fields $\mathbf{\mathrm{K}}_j=k_j\mathbf{\mathrm{I}}$ for $j=1,...,4$. We report the values of the physical parameters we use for this simulation in Table \ref{tab:3Dparam}.
\begin{table}[t]
	\centering
	\begin{tabular}{|c|r l|c|r l|}
		\hline
		\textbf{Parameter} & \multicolumn{2}{c|}{\textbf{Value}} &	\textbf{Parameter} & \multicolumn{2}{c|}{\textbf{Value}} \\ 
		\hline 
		$\rho$                  &  $1.00$       & $[\mathrm{Kg/m^3}]$ 
		& $k_j\;(j=1,...,4)$     &  $1.00$       & $[\mathrm{m^2}]$ \\ 
		\hline 
		$\lambda$               & $1.00$        & $[\mathrm{Pa}]$  
		& $\mu_j\;(j=1,...,4)$   & $1.00$        & $[\mathrm{Pa\cdot s}]$\\ 
		\hline 
		$\mu$                   & $1.00$        & $[\mathrm{Pa}]$  
		& $\beta_{12},\,\beta_{34}$	& $1.00$        & $[\mathrm{m^2/(N\cdot s)}]$  \\
		\hline 
		$\alpha_j\;(j=1,...,4)$  & $0.25$         & $[-]$  
		& $\beta_{13},\,\beta_{14},\,\beta_{23},\,\beta_{24}$	& $0.00$        & $[\mathrm{m^2/(N\cdot s)}]$  \\ 
		\hline 
		$c_j\;(j=1,...,4)$   	& $0.10$ & $[\mathrm{m^2/N}]$  &
		$\beta_j^{\mathrm{e}}\;(j=1,...,4)$
		& $0.00$ & $[\mathrm{m^2/(N\cdot s)}]$  \\  
		\hline 
	\end{tabular}
	\caption{Physical parameter values used in the 3D simulation.}
	\label{tab:3Dparam}
\end{table}

\begin{table}[t]
    \centering
    \begin{tabular}{c|c c c c|c c c c}
    \hline
    \multirow{2}{*}{$h$}
    & \multicolumn{4}{c|}{$\mathbb{P}_1-\mathbb{P}_2$} 
    & \multicolumn{4}{c}{$\mathbb{P}_1-\mathbb{P}_1$}
    \\
    & $||\boldsymbol{e}^{\boldsymbol{u}}||_{\mathrm{DG},e}$ 
    & $\mathrm{roc}^{\boldsymbol{u}}_{\mathrm{DG}}$ 
    & $\sum_{k\in J}||\sqrt{c_k}e^{p_{k}}||$
    & $\mathrm{roc}^{p}_{L^2}$  
    & $||\boldsymbol{e}^{\boldsymbol{u}}||_{\mathrm{DG},e}$ 
    & $\mathrm{roc}^{\boldsymbol{u}}_{\mathrm{DG}}$  
    & $\sum_{k\in J}||\sqrt{c_k}e^{p_{k}}||$
    & $\mathrm{roc}^{p}_{L^2}$
    \\ \hline
    0.866 
    & $1.97\times10^{-2}$ & -
    & $9.52\times10^{-3}$ & -
    & $7.26\times10^{-2}$ & -
    & $2.01\times10^{-2}$ & -
    \\ \hline
    0.433 
    & $3.69\times10^{-3}$ & $2.42$ 
    & $2.56\times10^{-3}$ & $1.89$ 
    & $2.87\times10^{-2}$ & $1.34$ 
    & $5.56\times10^{-3}$ & $1.85$ 
    \\ \hline
    0.217 
    & $6.53\times10^{-4}$ & $2.50$ 
    & $6.23\times10^{-4}$ & $2.04$ 
    & $1.06\times10^{-2}$ & $1.43$ 
    & $1.42\times10^{-3}$ & $1.94$ 
    \\ \hline
    0.108 
    & $1.44\times10^{-4}$ & $2.18$
    & $1.49\times10^{-4}$ & $2.06$ 
    & $4.43\times10^{-3}$ & $1.26$ 
    & $3.67\times10^{-4}$ & $1.98$
    \\ \hline \hline
    \multirow{2}{*}{$h$}
    & \multicolumn{4}{c|}{$\mathbb{P}_2-\mathbb{P}_3$} 
    & \multicolumn{4}{c}{$\mathbb{P}_2-\mathbb{P}_2$}
    \\
    & $||\boldsymbol{e}^{\boldsymbol{u}}||_{\mathrm{DG},e}$ 
    & $\mathrm{roc}^{\boldsymbol{u}}_{\mathrm{DG}}$ 
    & $\sum_{k\in J}||\sqrt{c_k}e^{p_{k}}||$
    & $\mathrm{roc}^{p}_{L^2}$  
    & $||\boldsymbol{e}^{\boldsymbol{u}}||_{\mathrm{DG},e}$ 
    & $\mathrm{roc}^{\boldsymbol{u}}_{\mathrm{DG}}$  
    & $\sum_{k\in J}||\sqrt{c_k}e^{p_{k}}||$
    & $\mathrm{roc}^{p}_{L^2}$
    \\ \hline
    0.866 
    & $4.26\times10^{-3}$ & -
    & $9.42\times10^{-4}$ & -
    & $1.97\times10^{-2}$ & -
    & $2.28\times10^{-3}$ & -
    \\ \hline
    0.433 
    & $3.91\times10^{-4}$ & $3.44$ 
    & $1.17\times10^{-4}$ & $3.00$ 
    & $3.69\times10^{-3}$ & $2.41$ 
    & $2.24\times10^{-4}$ & $3.34$ 
    \\ \hline
    0.217 
    & $4.02\times10^{-5}$ & $3.28$ 
    & $1.38\times10^{-5}$ & $3.10$ 
    & $6.54\times10^{-4}$ & $2.49$ 
    & $1.89\times10^{-5}$ & $3.56$ 
    \\ \hline
    0.108 
    & $6.23\times10^{-6}$ & $2.70$
    & $1.64\times10^{-6}$ & $3.07$ 
    & $1.47\times10^{-4}$ & $2.15$ 
    & $1.80\times10^{-7}$ & $3.38$
    \\ \hline  \hline
    \multirow{2}{*}{$h$}
    & \multicolumn{4}{c|}{$\mathbb{P}_3-\mathbb{P}_4$} 
    & \multicolumn{4}{c}{$\mathbb{P}_3-\mathbb{P}_3$}
    \\
    & $||\boldsymbol{e}^{\boldsymbol{u}}||_{\mathrm{DG},e}$ 
    & $\mathrm{roc}^{\boldsymbol{u}}_{\mathrm{DG}}$ 
    & $\sum_{k\in J}||\sqrt{c_k}e^{p_{k}}||$
    & $\mathrm{roc}^{p}_{L^2}$  
    & $||\boldsymbol{e}^{\boldsymbol{u}}||_{\mathrm{DG},e}$ 
    & $\mathrm{roc}^{\boldsymbol{u}}_{\mathrm{DG}}$  
    & $\sum_{k\in J}||\sqrt{c_k}e^{p_{k}}||$
    & $\mathrm{roc}^{p}_{L^2}$
    \\ \hline
    0.866 
    & $7.79\times10^{-4}$ & -
    & $1.88\times10^{-4}$ & -
    & $4.26\times10^{-3}$ & -
    & $2.34\times10^{-4}$ & -
    \\ \hline
    0.433 
    & $3.44\times10^{-5}$ & $4.50$ 
    & $1.17\times10^{-5}$ & $4.02$ 
    & $3.91\times10^{-4}$ & $3.44$ 
    & $1.95\times10^{-5}$ & $3.58$ 
    \\ \hline
    0.217 
    & $2.00\times10^{-6}$ & $4.10$ 
    & $7.31\times10^{-7}$ & $3.99$ 
    & $4.02\times10^{-5}$ & $3.28$ 
    & $1.37\times10^{-6}$ & $3.84$ 
    \\ \hline
    0.108 
    & $1.63\times10^{-7}$ & $3.62$
    & $4.59\times10^{-8}$ & $3.99$ 
    & $6.22\times10^{-6}$ & $2.70$ 
    & $8.80\times10^{-8}$ & $3.96$
    \\ \hline
    \end{tabular}
    \caption{Error estimates and convergence rates for the 3D test case.}
    \label{tab:errors3D}
\end{table}

\begin{figure}[t!]
    \begin{subfigure}[b]{0.5\textwidth}
          \resizebox{\textwidth}{!}{\definecolor{mycolor1}{rgb}{1.00000,1.00000,0.00000}%
\definecolor{mycolor2}{rgb}{0.00000,1.00000,1.00000}%
\pgfplotsset{
  log x ticks with fixed point/.style={
      xticklabel={
        \pgfkeys{/pgf/fpu=true}
        \pgfmathparse{exp(\tick)}%
        \pgfmathprintnumber[fixed  zerofill, precision=2]{\pgfmathresult}
        \pgfkeys{/pgf/fpu=false}
      }
  }
}
\begin{tikzpicture}

\begin{axis}[%
width=3.875in,
height=2.36in,
at={(2.6in,1.099in)},
scale only axis,
xmode=log,
xmin=0.10825,
xmax=0.866,
xminorticks=true,
xlabel = {$h$},
log x ticks with fixed point,
ymode=log,
ymin=6.98498691641807e-05,
ymax=0.0197481362791957,
yminorticks=true,
axis background/.style={fill=white},
title style={font=\bfseries},
title={Convergence for $\mathbb{P}_1-\mathbb{P}_2$ elements},
xmajorgrids,
xminorgrids,
ymajorgrids,
yminorgrids,
legend style={at={(0.02,0.56)}, anchor=south west, legend cell align=left, align=left, draw=white!15!black}
]
\addplot [color=blue, line width=2.0pt]
  table[row sep=crcr]{%
0.866	0.00465524789378733\\
0.433	0.00116266613416173\\
0.2165	0.000284900630050124\\
0.10825	6.98498691641807e-05\\
};
\addlegendentry{$||\sqrt{c_1}(p_1-p_1^\mathrm{ex})||$}

\addplot [color=red, line width=2.0pt]
  table[row sep=crcr]{%
0.866	0.0104099598691013\\
0.433	0.00289910116294254\\
0.2165	0.000701723565827211\\
0.10825	0.000166965254469331\\
};
\addlegendentry{$||\sqrt{c_2}(p_2-p_2^\mathrm{ex})||$}

\addplot [color=mycolor1, dashed, line width=2.0pt]
  table[row sep=crcr]{%
0.866	0.00465524789378622\\
0.433	0.00116266613416806\\
0.2165	0.000284900630012191\\
0.10825	6.98498691786793e-05\\
};
\addlegendentry{$||\sqrt{c_3}(p_3-p_3^\mathrm{ex})||$}

\addplot [color=green, dashed, line width=2.0pt]
  table[row sep=crcr]{%
0.866	0.0104099598690862\\
0.433	0.00289910116294314\\
0.2165	0.000701723565828025\\
0.10825	0.000166965254479653\\
};
\addlegendentry{$||\sqrt{c_4}(p_4-p_4^\mathrm{ex})||$}

\addplot [color=mycolor2, line width=2.0pt]
  table[row sep=crcr]{%
0.866	0.0197481362791957\\
0.433	0.00369212007243042\\
0.2165	0.000653290556293329\\
0.10825	0.000144209492584824\\
};
\addlegendentry{$||\boldsymbol{u}-\boldsymbol{u}^\mathrm{ex}||_\mathrm{DG,e}$}

\addplot [color=black, line width=1.5pt]
  table[row sep=crcr]{%
0.200	0.00020\\
0.141	0.00010\\
0.200	0.00010\\
0.200	0.00020\\
};

\node[right, align=left, text=black, font=\footnotesize]
at (axis cs:0.2,1.4e-4) {$2$};

\end{axis}
\end{tikzpicture}%
}
    \end{subfigure}%
    \begin{subfigure}[b]{0.5\textwidth}
        \resizebox{\textwidth}{!}{\definecolor{mycolor1}{rgb}{1.00000,1.00000,0.00000}%
\definecolor{mycolor2}{rgb}{0.00000,1.00000,1.00000}%
\pgfplotsset{
  log x ticks with fixed point/.style={
      xticklabel={
        \pgfkeys{/pgf/fpu=true}
        \pgfmathparse{exp(\tick)}%
        \pgfmathprintnumber[fixed  zerofill, precision=2]{\pgfmathresult}
        \pgfkeys{/pgf/fpu=false}
      }
  }
}
\begin{tikzpicture}

\begin{axis}[%
width=3.875in,
height=2.36in,
at={(2.6in,1.099in)},
scale only axis,
xmode=log,
xmin=0.10825,
xmax=0.866,
xminorticks=true,
xlabel = {$h$},
log x ticks with fixed point,
ymode=log,
ymin=0.0001,
ymax=0.1,
yminorticks=true,
axis background/.style={fill=white},
title style={font=\bfseries},
title={Convergence for $\mathbb{P}_1-\mathbb{P}_1$ elements},
xmajorgrids,
xminorgrids,
ymajorgrids,
yminorgrids,
legend style={legend cell align=left, align=left, draw=white!15!black}
]
\addplot [color=blue, line width=2.0pt]
  table[row sep=crcr]{%
0.866	0.0131414177476159\\
0.433	0.00356850220711797\\
0.2165	0.000939028290558287\\
0.10825	0.000240306808545034\\
};

\addplot [color=red, line width=2.0pt]
  table[row sep=crcr]{%
0.866	0.0185889532762586\\
0.433	0.00525478516122817\\
0.2165	0.00135666773867535\\
0.10825	0.000340306435084562\\
};

\addplot [color=mycolor1, dashed, line width=2.0pt]
  table[row sep=crcr]{%
0.866	0.0131414177476159\\
0.433	0.00356850220711797\\
0.2165	0.000939028290558287\\
0.10825	0.000240306808545034\\
};

\addplot [color=green, dashed, line width=2.0pt]
  table[row sep=crcr]{%
0.866	0.0185889532762586\\
0.433	0.00525478516122817\\
0.2165	0.00135666773867535\\
0.10825	0.000340306435084562\\
};

\addplot [color=mycolor2, line width=2.0pt]
  table[row sep=crcr]{%
0.866	0.0726916668593333\\
0.433	0.0287159386248445\\
0.2165	0.0106466329103709\\
0.10825	0.00442784909514489\\
};

\addplot [color=black, line width=1.5pt]
  table[row sep=crcr]{%
0.200	0.00060\\
0.141	0.00030\\
0.200	0.00030\\
0.200	0.00060\\
};

\node[right, align=left, text=black, font=\footnotesize]
at (axis cs:0.2,4e-4) {$2$};

\addplot [color=black, line width=1.5pt]
  table[row sep=crcr]{%
0.200	0.0080\\
0.141	0.00564\\
0.200	0.00564\\
0.200	0.0080\\
};

\node[right, align=left, text=black, font=\footnotesize]
at (axis cs:0.2,7e-3) {$1$};

\end{axis}
\end{tikzpicture}%
}
    \end{subfigure}%
    \\
    \begin{subfigure}[b]{0.5\textwidth}
          \resizebox{\textwidth}{!}{\definecolor{mycolor1}{rgb}{1.00000,1.00000,0.00000}%
\definecolor{mycolor2}{rgb}{0.00000,1.00000,1.00000}%
\pgfplotsset{
  log x ticks with fixed point/.style={
      xticklabel={
        \pgfkeys{/pgf/fpu=true}
        \pgfmathparse{exp(\tick)}%
        \pgfmathprintnumber[fixed  zerofill, precision=2]{\pgfmathresult}
        \pgfkeys{/pgf/fpu=false}
      }
  }
}
\begin{tikzpicture}

\begin{axis}[%
width=3.875in,
height=2.36in,
at={(2.6in,1.099in)},
scale only axis,
xmode=log,
xmin=0.10825,
xmax=0.866,
xminorticks=true,
xlabel = {$h$},
log x ticks with fixed point,
ymode=log,
ymin=1e-07,
ymax=0.01,
yminorticks=true,
axis background/.style={fill=white},
title style={font=\bfseries},
title={Convergence for $\mathbb{P}_2-\mathbb{P}_3$ elements},
xmajorgrids,
xminorgrids,
ymajorgrids,
yminorgrids,
legend style={legend cell align=left, align=left, draw=white!15!black}
]
\addplot [color=blue, line width=2.0pt]
  table[row sep=crcr]{%
0.866	0.00037116677525922\\
0.433	4.4097917409646e-05\\
0.2165	5.19328736659226e-06\\
0.10825	6.28935318455611e-07\\
};

\addplot [color=red, line width=2.0pt]
  table[row sep=crcr]{%
0.866	0.00112000971148612\\
0.433	0.000142254817339743\\
0.2165	1.65915572450037e-05\\
0.10825	1.97127724494002e-06\\
};

\addplot [color=mycolor1, dashed, line width=2.0pt]
  table[row sep=crcr]{%
0.866	0.000371166775252598\\
0.433	4.40979174079917e-05\\
0.2165	5.19328736302924e-06\\
0.10825	6.28935317092646e-07\\
};

\addplot [color=green, dashed, line width=2.0pt]
  table[row sep=crcr]{%
0.866	0.00112000971149109\\
0.433	0.000142254817340713\\
0.2165	1.65915572403352e-05\\
0.10825	1.97127725858742e-06\\
};

\addplot [color=mycolor2, line width=2.0pt]
  table[row sep=crcr]{%
0.866	0.00425642377410364\\
0.433	0.000390888740156458\\
0.2165	4.0199218317584e-05\\
0.10825	6.22645165383717e-06\\
};

\addplot [color=black, line width=1.5pt]
  table[row sep=crcr]{%
0.200	0.00000285\\
0.141	0.0000010\\
0.200	0.0000010\\
0.200	0.00000285\\
};

\node[right, align=left, text=black, font=\footnotesize]
at (axis cs:0.2,1.6e-6) {$3$};

\end{axis}
\end{tikzpicture}%
}
    \end{subfigure}%
    \begin{subfigure}[b]{0.5\textwidth}
        \resizebox{\textwidth}{!}{\definecolor{mycolor1}{rgb}{1.00000,1.00000,0.00000}%
\definecolor{mycolor2}{rgb}{0.00000,1.00000,1.00000}%
\pgfplotsset{
  log x ticks with fixed point/.style={
      xticklabel={
        \pgfkeys{/pgf/fpu=true}
        \pgfmathparse{exp(\tick)}%
        \pgfmathprintnumber[fixed  zerofill, precision=2]{\pgfmathresult}
        \pgfkeys{/pgf/fpu=false}
      }
  }
}

\begin{tikzpicture}

\begin{axis}[%
width=3.875in,
height=2.36in,
at={(2.6in,1.099in)},
scale only axis,
xlabel = {$h$},
log x ticks with fixed point,
xmode=log,
xmin=0.10825,
xmax=0.866,
xminorticks=true,
ymode=log,
ymin=1e-07,
ymax=0.1,
yminorticks=true,
axis background/.style={fill=white},
title style={font=\bfseries},
title={Convergence for $\mathbb{P}_2-\mathbb{P}_2$ elements},
xmajorgrids,
xminorgrids,
ymajorgrids,
yminorgrids,
legend style={legend cell align=left, align=left, draw=white!15!black}
]
\addplot [color=blue, line width=2.0pt]
  table[row sep=crcr]{%
0.866	0.00163260872752766\\
0.433	0.000152542812872126\\
0.2165	1.09707267379105e-05\\
0.10825	8.25799561468643e-07\\
};

\addplot [color=red, line width=2.0pt]
  table[row sep=crcr]{%
0.866	0.00197734113708458\\
0.433	0.000202026553729202\\
0.2165	1.89901024931966e-05\\
0.10825	2.03544431136528e-06\\
};

\addplot [color=mycolor1, dashed, line width=2.0pt]
  table[row sep=crcr]{%
0.866	0.00163260872752766\\
0.433	0.000152542812872126\\
0.2165	1.09707267379105e-05\\
0.10825	8.25799561468643e-07\\
};

\addplot [color=green, dashed, line width=2.0pt]
  table[row sep=crcr]{%
0.866	0.00197734113708458\\
0.433	0.000202026553729202\\
0.2165	1.89901024931966e-05\\
0.10825	2.03544431136528e-06\\
};

\addplot [color=mycolor2, line width=2.0pt]
  table[row sep=crcr]{%
0.866	0.0197457991206114\\
0.433	0.00369189696822007\\
0.2165	0.000654842461676569\\
0.10825	0.000146925530958431\\
};

\addplot [color=black, line width=1.5pt]
  table[row sep=crcr]{%
0.200	0.00040\\
0.141	0.00020\\
0.200	0.00020\\
0.200	0.00040\\
};

\node[right, align=left, text=black, font=\footnotesize]
at (axis cs:0.2,2.7e-4) {$2$};

\addplot [color=black, line width=1.5pt]
  table[row sep=crcr]{%
0.200	4.275e-6\\
0.141	1.5e-6\\
0.200	1.5e-6\\
0.200	4.275e-6\\
};

\node[right, align=left, text=black, font=\footnotesize]
at (axis cs:0.2,2.7e-6) {$3$};

\end{axis}
\end{tikzpicture}%
}
    \end{subfigure}%
        \\
    \begin{subfigure}[b]{0.5\textwidth}
        \resizebox{\textwidth}{!}{\definecolor{mycolor1}{rgb}{1.00000,1.00000,0.00000}%
\definecolor{mycolor2}{rgb}{0.00000,1.00000,1.00000}%

\pgfplotsset{
  log x ticks with fixed point/.style={
      xticklabel={
        \pgfkeys{/pgf/fpu=true}
        \pgfmathparse{exp(\tick)}%
        \pgfmathprintnumber[fixed  zerofill, precision=2]{\pgfmathresult}
        \pgfkeys{/pgf/fpu=false}
      }
  }
}
  
\begin{tikzpicture}

\begin{axis}[%
width=3.875in,
height=2.36in,
at={(2.6in,1.099in)},
scale only axis,
xmode=log,
xmin=0.10825,
xmax=0.866,
xminorticks=true,
xlabel = {$h$},
log x ticks with fixed point,
ymode=log,
ymin=1e-09,
ymax=0.001,
yminorticks=true,
axis background/.style={fill=white},
title style={font=\bfseries},
title={Convergence for $\mathbb{P}_3-\mathbb{P}_4$ elements},
xmajorgrids,
xminorgrids,
ymajorgrids,
yminorgrids,
legend style={legend cell align=left, align=left, draw=white!15!black}
]
\addplot [color=blue, line width=2.0pt]
  table[row sep=crcr]{%
0.866	4.78646788509183e-05\\
0.433	2.30997561192004e-06\\
0.2165	1.38592878703468e-07\\
0.10825	8.64153806147913e-09\\
};

\addplot [color=red, line width=2.0pt]
  table[row sep=crcr]{%
0.866	0.000250362295546197\\
0.433	1.60932622025708e-05\\
0.2165	1.01700847576806e-06\\
0.10825	6.40388520733158e-08\\
};

\addplot [color=mycolor1, dashed, line width=2.0pt]
  table[row sep=crcr]{%
0.866	4.78646788509183e-05\\
0.433	2.30997561192004e-06\\
0.2165	1.38592878703468e-07\\
0.10825	8.64153806147913e-09\\
};

\addplot [color=green, dashed, line width=2.0pt]
  table[row sep=crcr]{%
0.866	0.000250362295546197\\
0.433	1.60932622025708e-05\\
0.2165	1.01700847576806e-06\\
0.10825	6.40388520733158e-08\\
};

\addplot [color=mycolor2, line width=2.0pt]
  table[row sep=crcr]{%
0.866	0.000779425150632801\\
0.433	3.44245405609284e-05\\
0.2165	2.00311847641031e-06\\
0.10825	1.6272556029528e-07\\
};

\addplot [color=black, line width=1.5pt]
  table[row sep=crcr]{%
0.200	4.e-8\\
0.141	1.e-8\\
0.200	1.e-8\\
0.200	4.e-8\\
};

\node[right, align=left, text=black, font=\footnotesize]
at (axis cs:0.2,2e-8) {$4$};

\end{axis}
\end{tikzpicture}%
}
    \end{subfigure}%
    \begin{subfigure}[b]{0.5\textwidth}
        \resizebox{\textwidth}{!}{\definecolor{mycolor1}{rgb}{1.00000,1.00000,0.00000}%
\definecolor{mycolor2}{rgb}{0.00000,1.00000,1.00000}%

\pgfplotsset{
  log x ticks with fixed point/.style={
      xticklabel={
        \pgfkeys{/pgf/fpu=true}
        \pgfmathparse{exp(\tick)}%
        \pgfmathprintnumber[fixed  zerofill, precision=2]{\pgfmathresult}
        \pgfkeys{/pgf/fpu=false}
      }
  }
}
  
\begin{tikzpicture}

\begin{axis}[%
width=3.875in,
height=2.36in,
at={(2.6in,1.099in)},
scale only axis,
xmode=log,
xmin=0.10825,
xmax=0.866,
xminorticks=true,
xlabel = {$h$},
log x ticks with fixed point,
ymode=log,
ymin=1e-08,
ymax=0.01,
yminorticks=true,
axis background/.style={fill=white},
title style={font=\bfseries},
title={Convergence for $\mathbb{P}_3-\mathbb{P}_3$ elements},
xmajorgrids,
xminorgrids,
ymajorgrids,
yminorgrids,
legend style={legend cell align=left, align=left, draw=white!15!black}
]
\addplot [color=blue, line width=2.0pt]
  table[row sep=crcr]{%
0.866	0.000110591044109455\\
0.433	9.61221892921952e-06\\
0.2165	7.00544933816919e-07\\
0.10825	4.55688163961507e-08\\
};

\addplot [color=red, line width=2.0pt]
  table[row sep=crcr]{%
0.866	0.000258619566063748\\
0.433	2.12795801403297e-05\\
0.2165	1.46003197499181e-06\\
0.10825	9.35829628772726e-08\\
};

\addplot [color=mycolor1, dashed, line width=2.0pt]
  table[row sep=crcr]{%
0.866	0.000110591044109455\\
0.433	9.61221892921952e-06\\
0.2165	7.00544933816919e-07\\
0.10825	4.55688163961507e-08\\
};

\addplot [color=green, dashed, line width=2.0pt]
  table[row sep=crcr]{%
0.866	0.000258619566063748\\
0.433	2.12795801403297e-05\\
0.2165	1.46003197499181e-06\\
0.10825	9.35829628772726e-08\\
};

\addplot [color=mycolor2, line width=2.0pt]
  table[row sep=crcr]{%
0.866	0.00425581540416215\\
0.433	0.000390853883294224\\
0.2165	4.03016734332472e-05\\
0.10825	6.21668456485365e-06\\
};

\addplot [color=black, line width=1.5pt]
  table[row sep=crcr]{%
0.200	2.4e-5\\
0.141	8.e-6\\
0.200	8.e-6\\
0.200	2.5e-5\\
};

\node[right, align=left, text=black, font=\footnotesize]
at (axis cs:0.2,1.2e-5) {$3$};

\addplot [color=black, line width=1.5pt]
  table[row sep=crcr]{%
0.200	3.2e-7\\
0.141	8.e-8\\
0.200	8.e-8\\
0.200	3.2e-7\\
};

\node[right, align=left, text=black, font=\footnotesize]
at (axis cs:0.2,1.5e-7) {$4$};

\end{axis}
\end{tikzpicture}%
}
    \end{subfigure}%
    \caption{Test case 1: computed errors and convergence rates.}
    \label{fig:errors3D}
\end{figure}
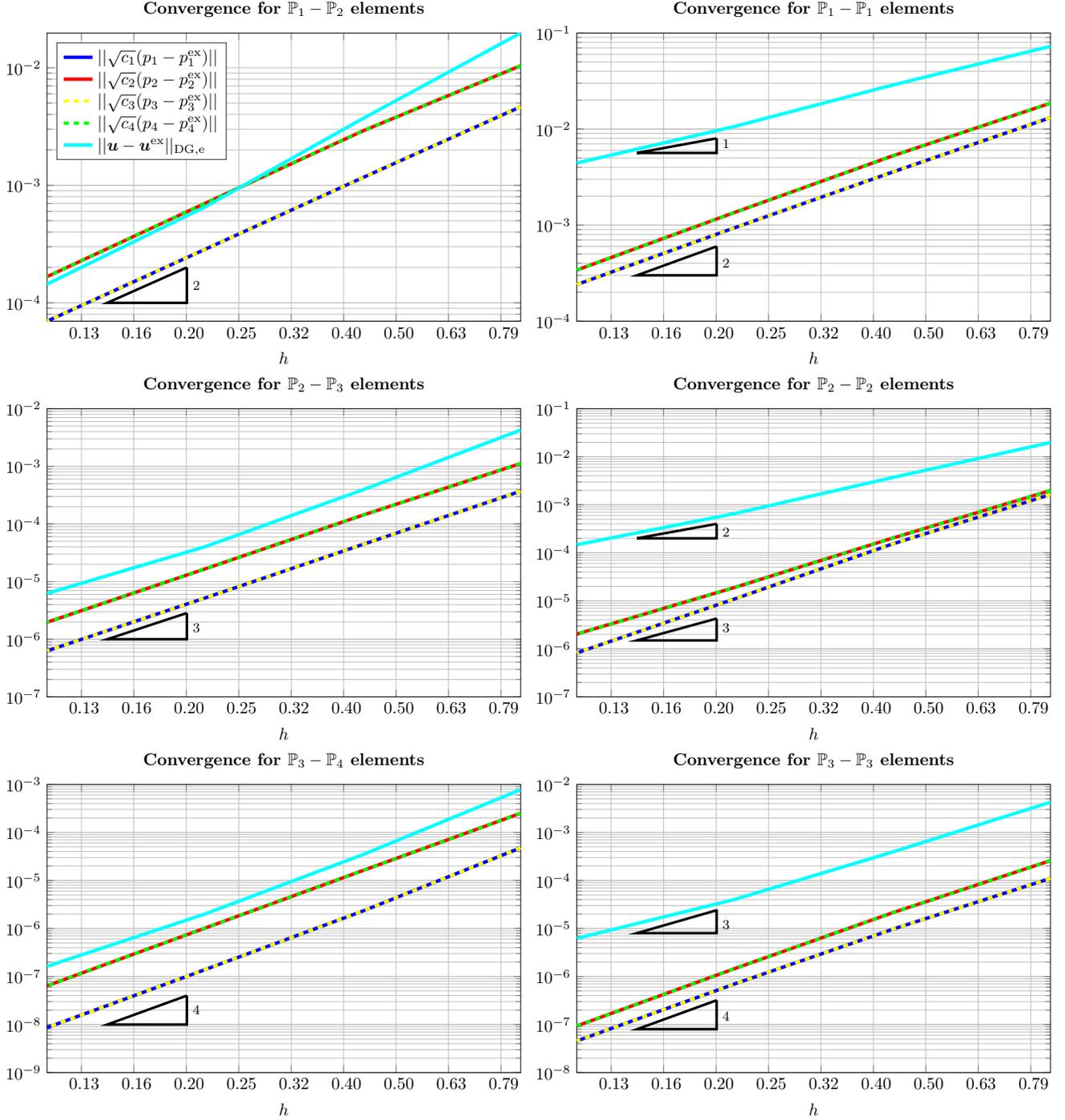
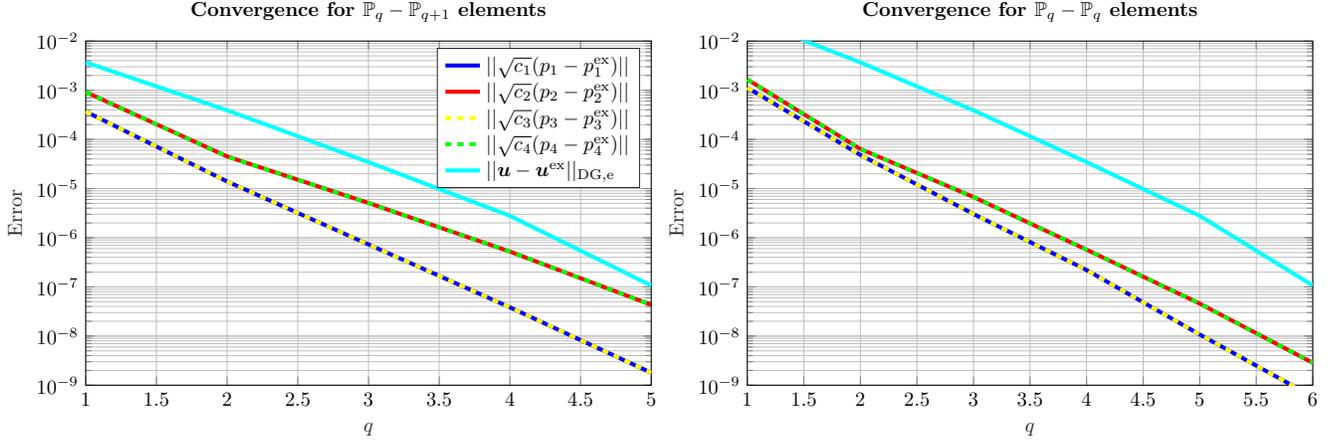
\begin{figure}[t]
    \begin{subfigure}[b]{0.5\textwidth}
          \resizebox{\textwidth}{!}{\definecolor{mycolor1}{rgb}{1.00000,1.00000,0.00000}%
\definecolor{mycolor2}{rgb}{0.00000,1.00000,1.00000}%

\begin{tikzpicture}
\begin{axis}[%
width=3.875in,
height=2.36in,
at={(1.733in,0.687in)},
scale only axis,
xmin=1,
xmax=5,
xlabel style={font=\color{white!15!black}},
xlabel={$q$},
ymode=log,
ymin=1e-09,
ymax=0.01,
yminorticks=true,
ylabel style={font=\color{white!15!black}},
ylabel={Error},
axis background/.style={fill=white},
title style={font=\bfseries},title={Convergence for $\mathbb{P}_q-\mathbb{P}_{q+1}$ elements},
xmajorgrids,
ymajorgrids,
yminorgrids,
legend style={legend cell align=left, align=left, draw=white!15!black}
]

\addplot [color=blue, line width=2.0pt]
  table[row sep=crcr]{%
1	0.000367667314229397\\
2	1.39449859084474e-05\\
3	7.30478427310854e-07\\
4	3.83633420940584e-08\\
5	1.79205436963207e-09\\
};
\addlegendentry{$||\sqrt{c_1}(p_1-p_1^\mathrm{ex})||$}

\addplot [color=red, line width=2.0pt]
  table[row sep=crcr]{%
1	0.000916776284214138\\
2	4.49849230924801e-05\\
3	5.08913635424217e-06\\
4	5.2270277875314e-07\\
5	4.31567355687846e-08\\
};
\addlegendentry{$||\sqrt{c_2}(p_2-p_2^\mathrm{ex})||$}

\addplot [color=mycolor1, dashed, line width=2.0pt]
  table[row sep=crcr]{%
1	0.000367667314229397\\
2	1.39449859084474e-05\\
3	7.30478427310854e-07\\
4	3.83633420940584e-08\\
5	1.79205436963207e-09\\
};
\addlegendentry{$||\sqrt{c_3}(p_3-p_3^\mathrm{ex})||$}

\addplot [color=green, dashed, line width=2.0pt]
  table[row sep=crcr]{%
1	0.000916776284214138\\
2	4.49849230924801e-05\\
3	5.08913635424217e-06\\
4	5.2270277875314e-07\\
5	4.31567355687846e-08\\
};
\addlegendentry{$||\sqrt{c_4}(p_4-p_4^\mathrm{ex})||$}

\addplot [color=mycolor2, line width=2.0pt]
  table[row sep=crcr]{%
1	0.00369212007243042\\
2	0.000390888740156458\\
3	3.44245405609284e-05\\
4	2.7960386190559e-06\\
5	1.06549047051629e-07\\
};

\addlegendentry{$||\boldsymbol{u}-\boldsymbol{u}^\mathrm{ex}||_\mathrm{DG,e}$}

\end{axis}
\end{tikzpicture}%
}
    \end{subfigure}%
    \begin{subfigure}[b]{0.5\textwidth}
        \resizebox{\textwidth}{!}{\definecolor{mycolor1}{rgb}{1.00000,1.00000,0.00000}%
\definecolor{mycolor2}{rgb}{0.00000,1.00000,1.00000}%

\begin{tikzpicture}
\begin{axis}[%
width=3.875in,
height=2.36in,
at={(1.733in,0.687in)},
scale only axis,
xmin=1,
xmax=6,
xlabel style={font=\color{white!15!black}},
xlabel={$q$},
ymode=log,
ymin=1e-09,
ymax=0.01,
yminorticks=true,
ylabel style={font=\color{white!15!black}},
ylabel={Error},
axis background/.style={fill=white},
title style={font=\bfseries},title={Convergence for $\mathbb{P}_q-\mathbb{P}_q$ elements},
xmajorgrids,
ymajorgrids,
yminorgrids,
legend style={legend cell align=left, align=left, draw=white!15!black}
]

\addplot [color=blue, line width=2.0pt]
  table[row sep=crcr]{%
1	0.00112845948098307\\
2	4.82382729364769e-05\\
3	3.03965051845185e-06\\
4	2.20158232403993e-07\\
5	1.07610341039862e-08\\
6   5.8430102909049e-10\\
};

\addplot [color=red, line width=2.0pt]
  table[row sep=crcr]{%
1	0.00166170897243361\\
2	6.38864057618663e-05\\
3	6.72919408955274e-06\\
4	5.70859635091825e-07\\
5	4.57606540833087e-08\\
6   2.84961425004106e-09\\
};

\addplot [color=mycolor1, dashed, line width=2.0pt]
  table[row sep=crcr]{%
1	0.00112845948098307\\
2	4.82382729364769e-05\\
3	3.03965051845185e-06\\
4	2.20158232403993e-07\\
5	1.07610341039862e-08\\
6   5.8430102909049e-10\\
};

\addplot [color=green, dashed, line width=2.0pt]
  table[row sep=crcr]{%
1	0.00166170897243361\\
2	6.38864057618663e-05\\
3	6.72919408955274e-06\\
4	5.70859635091825e-07\\
5	4.57606540833087e-08\\
6   2.84961425004106e-09\\
};

\addplot [color=mycolor2, line width=2.0pt]
  table[row sep=crcr]{%
1	0.0287159386248445\\
2	0.00369189696822007\\
3	0.000390853883294224\\
4	3.44476290335661e-05\\
5	2.7974495845470282e-06\\
6   1.069830441387346e-07\\
};


\end{axis}
\end{tikzpicture}%
}
    \end{subfigure}%
    \caption{Test case 1: computed errors against the order of DG approximation.}
    \label{fig:errorsp3D}
\end{figure}
\par
In Table \ref{tab:errors3D} we report the computed errors in both the DG and $L^2$ norms, together with the computed rates of convergence (roc) as a function of the mesh size $h$. The results reported in Table \ref{tab:errors3D} left (right) have been obtained with polynomials of degree $q=1,2,3$ for the pressure fields and with polynomials of degree $q$ ($q+1$) for the displacement one. We observe that the theoretical rates of convergence are respected both in the cases of $\mathbb{P}_q+\mathbb{P}_{q+1}$ elements and $\mathbb{P}_q+\mathbb{P}_{q}$ ones. Indeed, the rate of convergence of the displacement in $\mathrm{DG}-$norm is exactly the degree of approximation of the displacement in all the cases. At the same time, the $L^2-$norm rates of convergence for the pressures are equal to $q+1$. This fact is coherent with our energy stability estimate, in the first case, while we observe a superconvergence in the case of $L^2-$norm in pressure with $\mathbb{P}_q-\mathbb{P}_q$ elements. The rate of convergence can be observed also in Figure \ref{fig:errors3D}.
\par
A convergence analysis concerning the order of discretization is also performed. The results are reported in Figure \ref{fig:errorsp3D}.
\subsection{Test case 2: convergence analysis on 2D polygonal grids}
\begin{table}[t]
	\centering
	\begin{tabular}{|c|r l|c|r l|}
		\hline
		\textbf{Parameter} & \multicolumn{2}{c|}{\textbf{Value}} &	\textbf{Parameter} & \multicolumn{2}{c|}{\textbf{Value}} \\ 
		\hline 
		$\rho$                  &  $1000.00$       & $[\mathrm{Kg/m^3}]$ 
		& $k_1=k_2$    &  $3.50\times10^{-11}$          & $[\mathrm{m^2}]$ \\ 
		\hline 
		$\lambda$               & $505.00$        & $[\mathrm{Pa}]$  
		& $\mu$         & $216.00$        & $[\mathrm{Pa}]$\\ 
		\hline 
		 $\mu_1=\mu_2$                  & $3.50\times10^{-3}$        & $[\mathrm{Pa\cdot s}]$  
		& $\beta_{12}$	    & $10^{-7}$        & $[\mathrm{m^2/(N\cdot s)}]$  \\
		\hline 
		  $\alpha_1$        & $0.49$         & $[-]$  
		& $\alpha_2$	    & $0.51$        & $[-]$  \\ 
		\hline 
		$c_1=c_2$   	    & $10^{-6}$ & $[\mathrm{m^2/N}]$  &
		$\beta_1^{\mathrm{e}}=\beta_2^{\mathrm{e}}$
		& $0.00$ & $[\mathrm{m^2/(N\cdot s)}]$  \\  
		\hline 
	\end{tabular}
	\caption{Physical parameter values used in the 2D brain simulation.}
	\label{tab:2Dparam}
\end{table}
\begin{figure}[t!]
	\centering
	{\includegraphics[width=\textwidth]{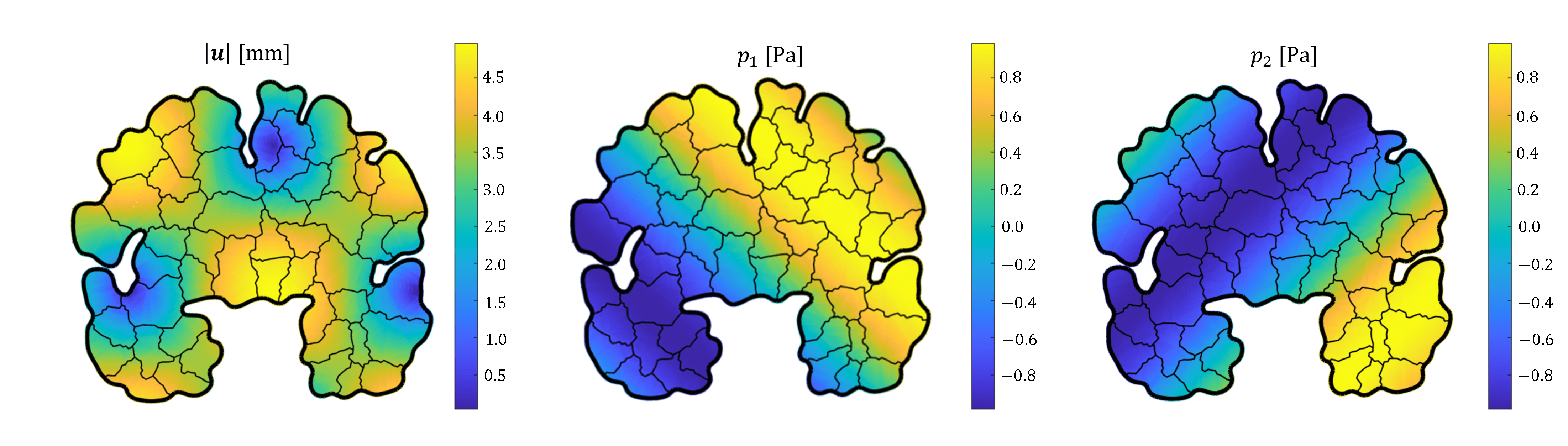}}
	\caption{Test case 2: computed solution (PolyDG of order 6) at the final time}
	\label{fig:AggSol}
\end{figure}
The first numerical test we perform is a two-dimensional setting on a polygonal agglomerated grid. Starting from structural Magnetic Resonance Images (MRI) of a brain from the OASIS-3 database \cite{OASIS3} we segment the brain by means of Freesurfer \cite{Freesurfer,Freesurfer:site}. After that, we construct a mesh of a slice of the brain along the frontal plane by means of VMTK \cite{VMTK:antiga}.
\par
The triangular resulting mesh is composed of $14\,372$ triangles. However, the generality of the PolyDG method allows us to use mesh elements of any shape, for this reason, we agglomerate the mesh by using ParMETIS \cite{Parmetis} and we obtain a polygonal mesh of 51 elements, as shown in Figure~\ref{fig:AggSol}. This mesh is then used to perform a convergence analysis, by varying the polynomial order.
\par
To test the convergence we consider the following exact solution for a case with two pressure fields $(|J|=2)$:
\begin{equation}
    \begin{matrix}
    \boldsymbol{u}(x,y,t)=\sin(\pi t)
    \begin{bmatrix}
    -\cos(\pi x)\cos(\pi y) \\
    \sin(\pi x)\sin(\pi y) 
    \end{bmatrix},
    \\[8pt]
    p_1(x,y,t)=10^4\pi\sin(\pi t)(\cos(\pi y)\sin(\pi x) + \cos(\pi x)\sin(\pi y)), \\[8pt]
    p_2(x,y,t)=10^4\pi\sin(\pi t)(\cos(\pi y)\sin(\pi x) - \cos(\pi x)\sin(\pi y)).
    \end{matrix}
\end{equation}
Concerning the time discretization, we use a timestep $\Delta t = 10^{-7}$ and a maximum time $T = 10^{-5}$. The forcing terms are then constructed to fulfil the continuous problem. The simulation is performed considering isotropic permeability $\mathbf{K}_j = k_j \mathbf{I}$ for $j=1,2$ and the values of the physical parameters used in this simulation are reported in Table \ref{tab:2Dparam}. These values are chosen to be comparable in dimensions to the parameters used in patient-specific simulations in literature \cite{guoSubjectspecificMultiporoelasticModel2018,vardakisFluidStructureInteraction2019}.
\par
In Figure~\ref{fig:AggSol}, we report the computed solution using the PolyDG of order 6 both in displacement and pressures at the final time. We can notice that the exact solution is smoothly approximated using the polygonal mesh. Unless the mesh contains few elements we are able to achieve the solution. 
\par
We report in Figure~\ref{fig:AggConv}, the convergence results for this test case. We can observe a spectral convergence increasing the polynomial order $q$. Finally, we observe that after $q=5$ we arrive to the lower bound of the pressure errors. This is due to the temporal discretization timestep we choose and it is coherent with the theory of space-time discretization errors.

\begin{figure}[t!]
    \centering
    \resizebox{0.75\textwidth}{!}{\definecolor{mycolor1}{rgb}{1.00000,1.00000,0.00000}%
\definecolor{mycolor2}{rgb}{0.00000,1.00000,1.00000}%

\begin{tikzpicture}
\begin{axis}[%
width=4.875in,
height=2.15in,
at={(1.733in,0.687in)},
scale only axis,
xmin=2,
xmax=6,
xlabel style={font=\color{white!15!black}},
xlabel={$q$},
ymode=log,
ymin=1e-07,
ymax=0.05,
yminorticks=true,
ylabel style={font=\color{white!15!black}},
ylabel={Error},
axis background/.style={fill=white},
title style={font=\bfseries},title={Convergence for $\mathbb{P}_q-\mathbb{P}_{q}$ elements},
xmajorgrids,
ymajorgrids,
yminorgrids,
legend style={legend cell align=left, align=left, draw=white!15!black},
xtick={2,3,4,5,6}
]

\addplot [color=blue, line width=1.5pt]
  table[row sep=crcr]{%
1   0.036895833555724 \\    
2   0.008692795289993 \\    
3   0.001008998275595 \\    
4   5.94990000000e-05 \\    
5   5.59320000000e-06 \\    
6   5.89410000000e-06 \\    
};
\addlegendentry{$||p_{1}-p_{1}^\mathrm{ex}||$}

\addplot [color=red, line width=1.5pt]
  table[row sep=crcr]{%
1   0.038421530520484   \\   
2   0.006491376554168   \\   
3   0.001547955835478   \\   
4   7.42020000000e-05   \\   
5   6.44960000000e-06   \\   
6   5.79300000000e-06   \\   
};
\addlegendentry{$||p_{2}-p_{2}^\mathrm{ex}||$}

\addplot [color=mycolor2, line width=1.5pt]
  table[row sep=crcr]{%
1	1.894582847587361e-05 \\  
2	3.921420805640827e-05 \\  
3	1.228202992399341e-05 \\  
4	2.523400000000000e-06 \\  
5	2.924000000000000e-07 \\  
6   9.670000000000000e-08 \\  
};

\addlegendentry{$||\boldsymbol{u}-\boldsymbol{u}^\mathrm{ex}||_\mathrm{DG,e}$}

\end{axis}
\end{tikzpicture}%
}
    \caption{Test case 2: computed errors against the order of PolyDG approximation on the brain section.}
	\label{fig:AggConv}
\end{figure}
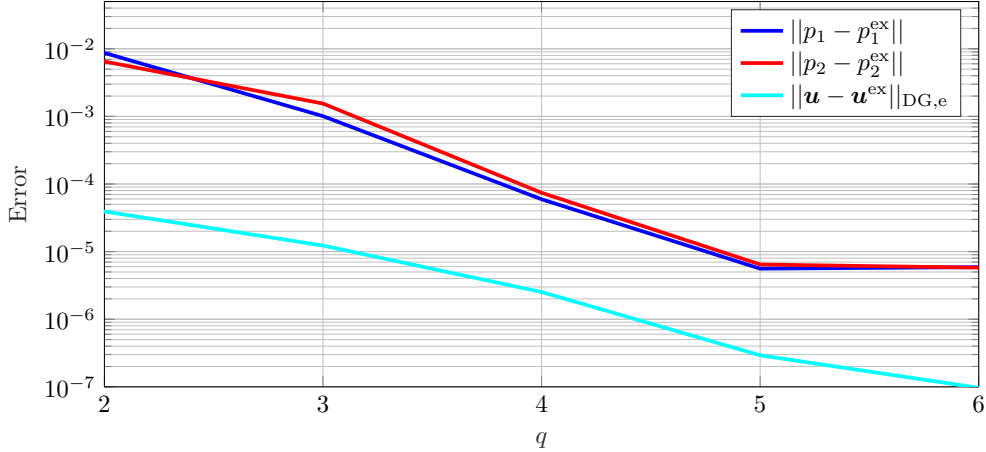

\subsection{Test case 3: simulation on a brain geometry}
Finally, we perform a three-dimensional simulation starting from a structural MRI from the OASIS-3 database \cite{OASIS3} and we segment it employing Freesurfer \cite{Freesurfer,Freesurfer:site} and 3DSlicer \cite{3DSilcer,3DSlicer:site}. Finally, the mesh is constructed using the SVMTK library \cite{Mardal:Mesh}. The tetrahedral resulting mesh is composed of 81'194 elements. The problem is solved by means of a code implemented in FEniCS finite element software \cite{FenicsCode} (version 2019).
\par
In this case we refer to the mathematical modelling of \cite{tullyCerebralWaterTransport2011}, which proposed the simulation of four different fluid networks: arterial blood ($\mathrm{A}$), capillary blood ($\mathrm{C}$), venous blood ($\mathrm{V}$) and cerebrospinal/extracellular fluid ($\mathrm{E}$). In this context, the boundary condition are constructed by dividing the boundary of the domain into the ventricular boundary $\Gamma_\mathrm{Vent}$ and the skull $\Gamma_\mathrm{Skull}$, as visible in Figure \ref{fig:mesh}.
\begin{figure}[t]
	\centering
	{\includegraphics[width=\textwidth]{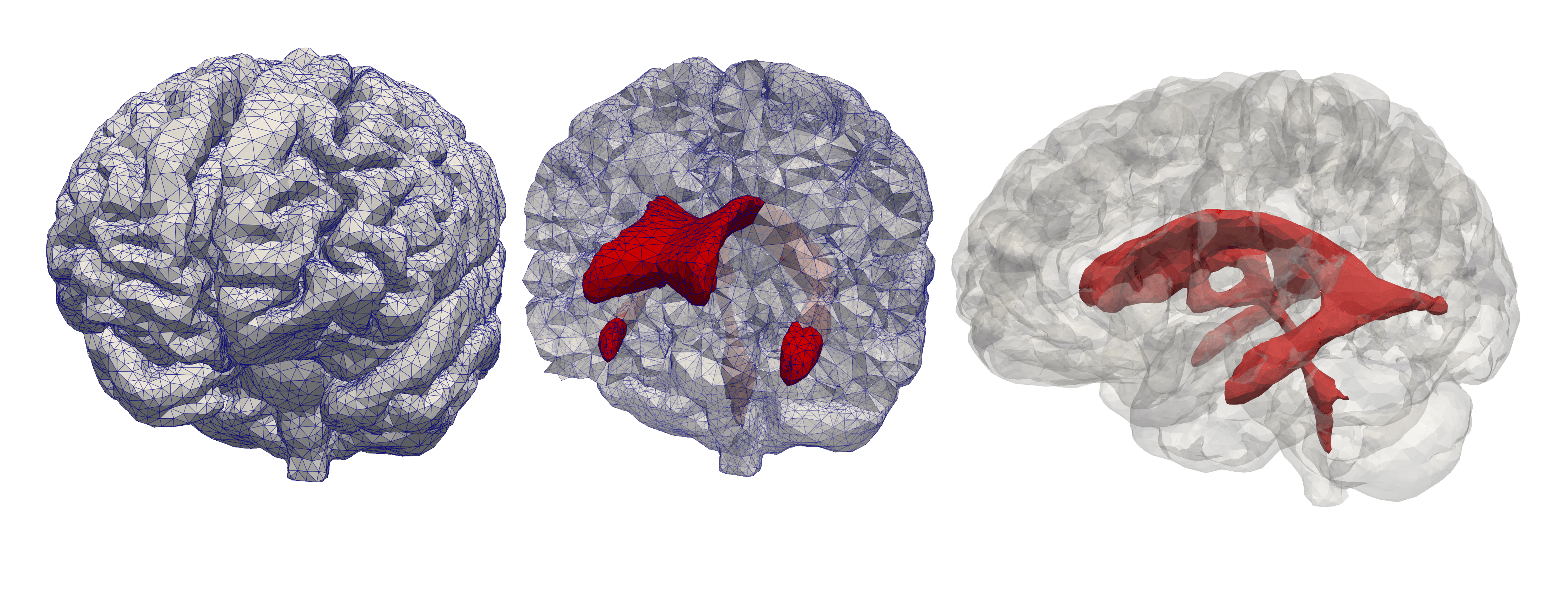}}
	\caption{Test case 3: brain 3D mesh. An external view of the mesh (on the left), an internal view with the ventricles in red (in the middle) and a visualization of ventricles boundary in red and skull in trasparency (on the right).}
	\label{fig:mesh}
\end{figure}
\par
The discretization is based on a DG method in space with polynomials of order 2 for each solution field. Moreover, we apply a temporal discretization with $\Delta t = 10^{-3}\;\mathrm{s}$ by considering a heartbeat of duration $1 \mathrm{s}$. Indeed, we apply periodic boundary conditions to the problem. Concerning the elastic movement, we consider a fixed skull, while the ventricles boundary can deform under the stress of the CSF inside the ventricles:
\begin{equation}
    \boldsymbol{u}=\boldsymbol{0}\quad\mathrm{on}\;\Gamma_\mathrm{Skull}, \qquad \boldsymbol{\sigma}_{\mathrm{E}}(\boldsymbol{u})\boldsymbol{n} - \sum_{j\in J} \alpha_j p_j \boldsymbol{n} = - \tilde{p}_\mathrm{E}^\mathrm{Vent} \boldsymbol{n}\quad\mathrm{on}\;\Gamma_\mathrm{Vent}.
\end{equation}
Concerning the arterial blood, we impose sinusoidal pressure on the skull, to mimic the pressure variations due to the heartbeat, with a mean value of $70 \mathrm{mmHg}$. At the same time, we do not allow inflow/outflow of blood from the ventricular boundary:
\begin{equation}
p_\mathrm{A}=70+10\sin(2\pi t)\;\mathrm{mmHg},\quad\mathrm{on}\;\Gamma_\mathrm{Skull}, \qquad \nabla p_\mathrm{A}\cdot\boldsymbol{n} = \boldsymbol{0},\quad\mathrm{on}\;\Gamma_\mathrm{Vent}.
\end{equation}
Concerning the capillary blood, we do not allow any inflow/outflow of blood from the boundary:
\begin{equation}
\nabla p_\mathrm{C}\cdot\boldsymbol{n} = \boldsymbol{0},\quad\mathrm{on}\;\partial \Omega.
\end{equation}
For the venous blood, we impose the pressure value at the boundary:
\begin{equation}
p_\mathrm{V} = 6\;\mathrm{mmHg},\quad\mathrm{on}\;\partial \Omega.
\end{equation}
Finally, we assume that the CSF can flow from the parenchyma to the ventricles and we impose a pulsatility around a baseline of $5\;\mathrm{mmHg}$. Indeed we impose:
\begin{equation}
p_\mathrm{E}=5+(2+0.012)\sin(2\pi t)\;\mathrm{mmHg}\quad\mathrm{on}\;\Gamma_\mathrm{Vent}, \qquad p_\mathrm{E}=5+2\sin(2\pi t)\;\mathrm{mmHg}\quad\mathrm{on}\;\Gamma_\mathrm{Skull}.
\end{equation}
\begin{figure}[t]
	\centering
	{\includegraphics[width=\textwidth]{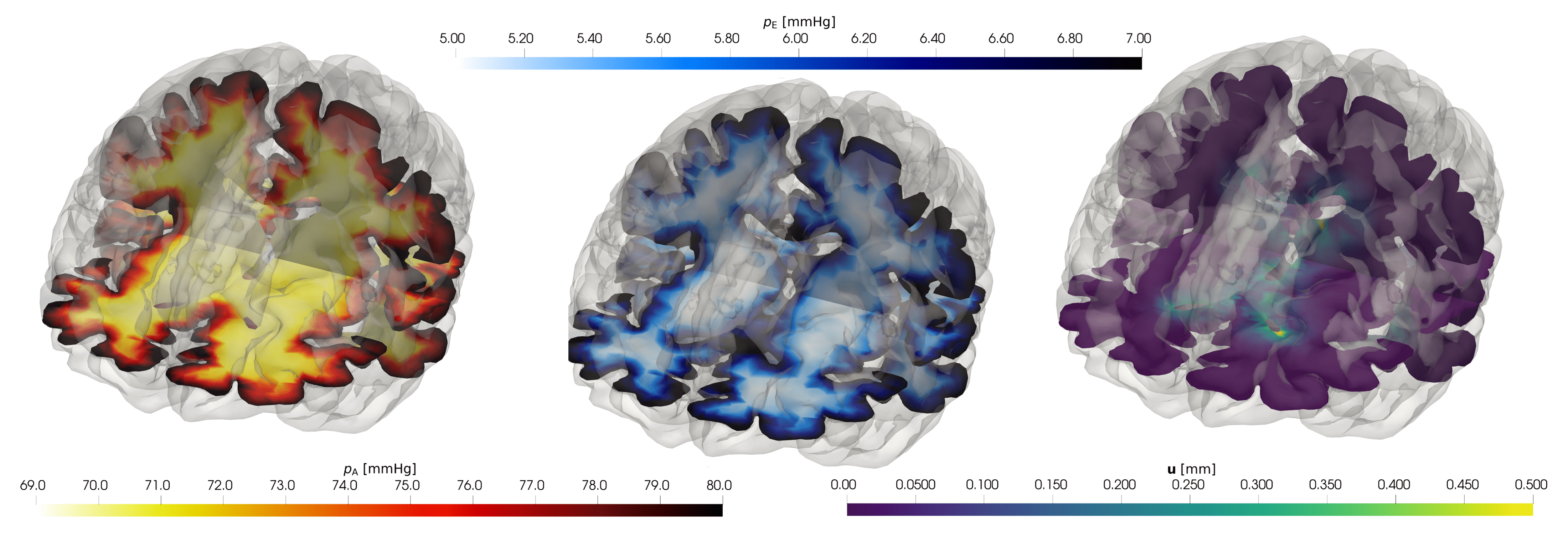}}
	\caption{Test case 3: solution of the MPET dynamic system in the patient-specific geometry at $t=0.25\;\mathrm{s}$. From left to right: $p_\mathrm{A}$, $p_\mathrm{E}$ and $|\boldsymbol{u}|$.}
	\label{fig:PSS}
\end{figure}
We add the discharge term to the venous pressure equation with a parameter $\beta_\mathrm{V}^\mathrm{e}=10^{-6}\;\mathrm{m^2/(N\cdot s)}$ and we consider an external veins pressure $\tilde{p}_\mathrm{Veins}=6\;\mathrm{mmHg}$. To solve the algebraic resulting problem, we apply a monolithic strategy, by using an iterative GMRES method, with a SOR preconditioner.
\par
In Figure \ref{fig:PSS} we report the numerical solution computed at time $t=0.25 \mathrm{s}$ with the parameters from \cite{lee_mixed_2019}. As we can observe, we obtain maximum values of displacement on the ventricular boundary. The pressure values obtained are coherent with the imposed boundary conditions and the largest gradients are related to the arterial compartment. Concerning venous and capillary pressures, we do not report the maps of values inside the brain. Indeed, the computed values are near to spatially constant at $6\;\mathrm{mmHg}$ and $38\;\mathrm{mmHg}$, respectively. This is coherent to what is found in similar studies \cite{vardakis_exploring_2020}.
\section{Conclusions}
\label{sec:conclusion}
In this work, we have introduced a numerical polyhedral discontinuous Galerkin method for the solution of the dynamic multiple-network poroelastic model in the dynamic version. We derived stability and convergence error estimates for arbitrary-order approximation. Moreover, we proposed a temporal discretization based on the coupling of Newmark $\beta$-method for the second order equation and $\theta$-method for the first order equations.
\par
The numerical convergence tests were presented both in two and three dimensions. In particular, we presented a test on a slice of brain, with an agglomerated polygonal mesh. These tests confirmed the theoretical results of our analysis and the possibility to use this formulation to solve the problem on coarse polygonal meshes. Finally, we performed a numerical simulation on a real 3D brain geometry.
\section*{Declaration of competing interests}
The authors declare that they have no known competing financial interests or personal relationships that could have appeared to influence the work reported in this article.

\section*{Acknowledgments}
PFA has been partially funded by the research grant PRIN2017 n. 201744KLJL. PFA, LD and AMQ has been partially funded by the research grant PRIN2020 n. 20204LN5N5 funded by the Italian Ministry of Universities and Research (MUR). MC, PFA, LD and AMQ are members of INdAM-GNCS. The brain MRI images were provided by OASIS-3: Longitudinal Multimodal Neuroimaging: Principal Investigators: T. Benzinger, D. Marcus, J. Morris; NIH P30 AG066444, P50 AG00561, P30 NS09857781, P01 AG026276, P01 AG003991, R01 AG043434, UL1 TR000448, R01 EB009352. AV-45 doses were provided by Avid Radiopharmaceuticals, a wholly owned subsidiary of Eli Lilly.

\appendix
\section{Derivation of the weak formulation}
\label{sec:wfder}
Considering the momentum equation we can introduce a test function $\boldsymbol{v}\in V$ and write the weak formulation:
\begin{equation}
    \int_{\Omega}\rho\dfrac{\partial^2\boldsymbol{u}}{\partial t^2}\cdot\boldsymbol{v}
    - \int_{\Omega}\nabla\cdot\boldsymbol{\sigma}_\mathrm{E}(\boldsymbol{u})\cdot\boldsymbol{v} +
    \sum_{k\in J} \int_{\Omega}\alpha_k\nabla p_k\cdot\boldsymbol{v} = \int_\Omega\boldsymbol{f}\cdot\boldsymbol{v} \qquad \forall\boldsymbol{v}\in V.
\end{equation}
Each component can be treated separately to reach:
\begin{equation*}
    \begin{split}
    \int_{\Omega}\rho\dfrac{\partial^2\boldsymbol{u}}{\partial t^2}\cdot\boldsymbol{v} = &
    \rho\left(\dfrac{\partial^2\boldsymbol{u}}{\partial t^2},\boldsymbol{v}\right)_\Omega, \\[12pt]
   \int_{\Omega}\nabla\cdot\boldsymbol{\sigma}_\mathrm{E}(\boldsymbol{u})\cdot\boldsymbol{v} = &
   -\int_{\Omega}\boldsymbol{\sigma}_\mathrm{E}(\boldsymbol{u}):\nabla\boldsymbol{v} + \int_{\partial\Omega}(\boldsymbol{\sigma}_\mathrm{E}(\boldsymbol{u})\cdot\boldsymbol{n})\cdot\boldsymbol{v}\mathrm{d}\sigma = \\
   = & -2\mu\int_{\Omega}\boldsymbol{\varepsilon}(\boldsymbol{u}):\nabla\boldsymbol{v} -\lambda\int_{\Omega}(\nabla\cdot\boldsymbol{u})(\nabla\cdot\boldsymbol{v}) + \int_{\Gamma_N} \boldsymbol{\sigma}_\mathrm{E}(\boldsymbol{u})\cdot\boldsymbol{n}\cdot\boldsymbol{v}\mathrm{d}\sigma \\
    = & -2\mu\left(\boldsymbol{\varepsilon}(\boldsymbol{u}),\boldsymbol{\varepsilon}(\boldsymbol{v})\right)_\Omega -\lambda\left(\nabla\cdot\boldsymbol{u},\nabla\cdot\boldsymbol{v}\right)_\Omega + (\boldsymbol{\sigma}_\mathrm{E}(\boldsymbol{u})\cdot\boldsymbol{n},\boldsymbol{v})_{\Gamma_N}, \\[12pt]
    \int_{\Omega}\alpha_k\nabla p_k\cdot\boldsymbol{v} = &
    - \int_{\Omega}\alpha_k p_k(\nabla\cdot\boldsymbol{v})
    + \int_{\partial\Omega}\alpha_k (p_k\boldsymbol{n}\cdot\boldsymbol{v})\mathrm{d}\sigma \\
    = & -\left(\alpha_k p_k,\nabla\cdot\boldsymbol{v} \right)_\Omega + \alpha_k (p_k\boldsymbol{n},\boldsymbol{v})_{\Gamma_N}, \\[12pt]
    \int_\Omega\boldsymbol{f}\cdot\boldsymbol{v} = (\boldsymbol{f},\boldsymbol{v})_\Omega.
    \end{split}
\end{equation*}
Finally, exploiting the Neumann boundary condition on the boundary $\Gamma_N$ for the momentum equation, we have:
\begin{equation*}
    (\boldsymbol{\sigma}_\mathrm{E}(\boldsymbol{u})\cdot\boldsymbol{n},\boldsymbol{v})_{\Gamma_N}-\sum_{k\in J}\alpha_k (p_k\boldsymbol{n},\boldsymbol{v})_{\Gamma_N} = \left(\boldsymbol{\sigma}_\mathrm{E}(\boldsymbol{u})\cdot\boldsymbol{n}-\sum_{k\in J}\alpha_k p_k\boldsymbol{n},\boldsymbol{v}\right)_{\Gamma_N} = \left(\boldsymbol{h}_{\boldsymbol{u}},\boldsymbol{v}\right)_{\Gamma_N},
\end{equation*}
this identity leads us to the final weak formulation of the momentum equation:
\begin{equation}
\label{eq:weakmom}
\begin{split}
     \rho\left(\dfrac{\partial^2\boldsymbol{u}}{\partial t^2},\boldsymbol{v}\right)_\Omega + 2\mu\left(\boldsymbol{\varepsilon}(\boldsymbol{u}),\boldsymbol{\varepsilon}(\boldsymbol{v})\right)_\Omega +\lambda\left(\nabla\cdot\boldsymbol{u},\nabla\cdot\boldsymbol{v}\right)_\Omega - \sum_{k \in J} & \left(\alpha_k p_k,\nabla\cdot\boldsymbol{v} \right)_\Omega = \\
     = & (\boldsymbol{f},\boldsymbol{v})_\Omega + 
     \left(\boldsymbol{h}_{\boldsymbol{u}},\boldsymbol{v}\right)_{\Gamma_N}
     \qquad\forall\boldsymbol{v}\in V.
\end{split}
\end{equation}
Substituting the definitions of the bilinear forms into the equation \eqref{eq:weakmom} we obtain:
\begin{equation}
\label{eq:abmom}
     \rho\left(\dfrac{\partial^2\boldsymbol{u}}{\partial t^2},\boldsymbol{v}\right)_\Omega + a(\boldsymbol{u},\boldsymbol{v}) - \sum_{k \in J} b_k(p_k,\boldsymbol{v}) = F(\boldsymbol{v})
     \qquad\forall\boldsymbol{v}\in V.
\end{equation}
\par
The other conservation equations in Equation \eqref{eq:fullprob} can be derived following the same procedure. For $j\in J$, we multiply by a test function $q_j\in Q_j$ and we integrate over the domain $\Omega$:
\begin{equation}
    c_j\int_{\Omega}\dfrac{\partial p_j}{\partial t} q_j +
    \int_{\Omega}\nabla\cdot\left(\alpha_j\dfrac{\partial \boldsymbol{u}}{\partial t} -
    \dfrac{\boldsymbol{\mathrm{K}}_j}{\mu_j}\nabla p_j\right) 
    + \sum_{k \in J} \int_{\Omega}\beta_{jk}(p_j-p_k)q_j + \int_{\Omega} + \beta^\mathrm{e}_j p_j q_j=\int_{\Omega} g_j q_j \qquad \forall q_j\in Q_j.
\end{equation}
We treat each component separately:
\begin{equation*}
\begin{split}
    \int_{\Omega}c_j\dfrac{\partial p_j}{\partial t} q_j =\, &
    \left(c_j\dfrac{\partial p_j}{\partial t},q_j\right)_{\Omega},
    \\[12pt]
    \int_{\Omega}\nabla\cdot\left(\alpha_j\dfrac{\partial \boldsymbol{u}}{\partial t}\right)q_j=\, & \alpha_j\left(\nabla\cdot\left(\dfrac{\partial \boldsymbol{u}}{\partial t}\right),q_j\right)_{\Omega},
    \\[12pt]
    \int_{\Omega}\nabla\cdot\left(
    \dfrac{\boldsymbol{\mathrm{K}}_j}{\mu_j}\nabla p_j\right)q_j = & - \int_{\Omega}    \dfrac{\boldsymbol{\mathrm{K}}_j}{\mu_j}\nabla p_j \cdot \nabla q_j  + \int_{\partial\Omega}    \left(\dfrac{\boldsymbol{\mathrm{K}}_j}{\mu_j}\nabla p_j\right) \cdot q_j\boldsymbol{n}\; \mathrm{d}\sigma = \\
    = & - \left(\dfrac{\boldsymbol{\mathrm{K}}_j}{\mu_j}\nabla p_j ,\nabla q_j\right)_\Omega  +  \left(\dfrac{\boldsymbol{\mathrm{K}}_j}{\mu_j}\nabla p_j \cdot \boldsymbol{n},q_j\right)_{\Gamma_N^j} = \\
    = & - \dfrac{\boldsymbol{\mathrm{K}}_j}{\mu_j}(\nabla p_j ,\nabla q_j)_\Omega  +(h_j,q_j)_{\Gamma_N^j},
    \\[12pt]
    \int_{\Omega}\beta_{jk}(p_j-p_k)q_j =\,& 
    (\beta_{jk}p_j,q_j)_\Omega - (\beta_{jk}p_k,q_j)_\Omega,
    \\[12 pt]
    \int_{\Omega}\beta_{j}^\mathrm{e} p_j q_j =\,& 
    (\beta_{j}^\mathrm{e} p_j,q_j)_\Omega,
    \\[12 pt]
    \int_{\Omega} g_j q_j = & (g_j,q_j)_\Omega.
\end{split}
\end{equation*}
We sum up the terms and we obtain the following equation:
\begin{equation}
\label{eq:weakpress}
\begin{split}
    c_j\left(\dfrac{\partial p_j}{\partial t},q_j\right)_{\Omega} + & \alpha_j\left(\nabla\cdot\left(\dfrac{\partial \boldsymbol{u}}{\partial t}\right),q_j\right)_{\Omega} + \dfrac{\boldsymbol{\mathrm{K}}_j}{\mu_j}(\nabla p_j ,\nabla q_j)_\Omega + \sum_{k \in J}(\beta_{jk}p_j,q_j)_\Omega -  
    \\ - &  \sum_{k \in J}(\beta_{jk}p_k,q_j)_\Omega + (\beta_{j}^\mathrm{e} p_j,q_j)_\Omega = (g_j,q_j)_\Omega + (h_j,q_j)_{\Gamma_N^j}\qquad \forall q_j\in Q_j \quad \forall j\in J.
\end{split}
\end{equation}
Introducing now the definitions of the bilinear forms, we can rewrite problem \eqref{eq:weakpress} in an abstract form:
\begin{equation}
\begin{split}
    c_j\left(\dfrac{\partial p_j}{\partial t},q_j\right)_{\Omega} + b_j\left(q_j,
    \dfrac{\partial \boldsymbol{u}}{\partial t}\right)
    + s_j(p_j,q_j) +   C_j\left((p_k)_{k\in J},q_j\right) = G_j(q_j) \quad \forall q_j\in Q_j\quad j\in J.
\end{split}
\end{equation}
Finally, the weak formulation of problem \eqref{eq:fullprob} reads:
\par
\bigskip
Find $\boldsymbol{u}(t)\in V$ and $q_j(t) \in Q_j$ with $j\in J$ such that $\forall t>0$:
\begin{equation}
\begin{dcases}
     \rho\left(\dfrac{\partial^2\boldsymbol{u}(t)}{\partial t^2},\boldsymbol{v}\right)_\Omega + a(\boldsymbol{u}(t),\boldsymbol{v}) - \sum_{k \in J} b_k(p_k(t),\boldsymbol{v}) = F(\boldsymbol{v}),
     & \forall\boldsymbol{v}\in V, \\[8pt]
     c_j\left(\dfrac{\partial p_j}{\partial t},q_j\right)_{\Omega} + b_j\left(q_j,
    \dfrac{\partial \boldsymbol{u}}{\partial t}\right)
    + s_j(p_j,q_j) +   C_j\left((p_k)_{k\in J},q_j\right) = G_j(q_j), & \forall q_j\in Q_j\quad j\in J, \\[8pt]
    \boldsymbol{u}(0)=\boldsymbol{u}_0, & \mathrm{in}\;\Omega, \\[8pt]
    \dfrac{\partial\boldsymbol{u}}{\partial t}(0)=\boldsymbol{v}_0, & \mathrm{in}\;\Omega, \\[8pt]
    p_j(0)=p_{j0}, & \mathrm{in}\;\Omega\quad j\in J,\\[8pt]
    \boldsymbol{u}(t) = \boldsymbol{u}_\mathrm{D}(t), & \mathrm{on}\;\Gamma_D,\\[8pt]
   q_j(t)=q^\mathrm{D}_j(t), & \mathrm{on}\;\Gamma_D^j \quad j\in J.
\end{dcases}
\end{equation}
\section{Derivation of PolyDG semi-discrete formulation}
\label{sec:PolyDGder}
In this section, we derive the PolyDG formulation of the Equation \eqref{eq:weakform}. First, we rewrite the momentum equation in a Discontinuous Galerkin (DG) framework. We proceed as usual to obtain:
\begin{equation*}
\begin{split}
   \int_{\Omega}\nabla\cdot\boldsymbol{\sigma}_\mathrm{E}(\boldsymbol{u}_h)\cdot\boldsymbol{v}_h = & 
   \sum_{K\in\partition} \int_{K}\nabla\cdot\boldsymbol{\sigma}_\mathrm{E}(\boldsymbol{u}_h)\cdot\boldsymbol{v}_h \\
   = & - \sum_{K\in\partition} \int_{K}\boldsymbol{\sigma}_\mathrm{E}(\boldsymbol{u}_h):\nabla\boldsymbol{v}_h + \sum_{K\in\partition} \int_{\partial K}(\boldsymbol{\sigma}_\mathrm{E}(\boldsymbol{u}_h)\cdot\boldsymbol{n})\cdot\boldsymbol{v}_h\mathrm{d}\sigma = \\
   = & - \int_{\Omega}\boldsymbol{\sigma}_\mathrm{E}(\boldsymbol{u}_h):\nabla_h\boldsymbol{v}_h +\sum_{F\in\facesinternal\cup\facesD}\int_{F}\averagel\boldsymbol{\sigma}_\mathrm{E}(\boldsymbol{u}_h)\averager : \jjumpl\boldsymbol{v}_h\jjumpr\mathrm{d}\sigma +  \sum_{F\in\facesinternal\cup\facesD}\int_{F}  \jjumpl\boldsymbol{u}_h\jjumpr : \averagel\boldsymbol{\sigma}_\mathrm{E}(\boldsymbol{v}_h)\averager\mathrm{d}\sigma + \\ 
   & -\sum_{F\in\facesinternal\cup\facesD} \int_{F}\eta \jjumpl\boldsymbol{u}_h\jjumpr : \jjumpl\boldsymbol{v}_h\jjumpr\mathrm{d}\sigma + \int_{\Gamma_N}(\boldsymbol{\sigma}_\mathrm{E}(\boldsymbol{u}_h)\cdot\boldsymbol{n})\cdot\boldsymbol{v}_h\mathrm{d}\sigma \qquad \forall \boldsymbol{v}_h\in\Vh 
\end{split}
\end{equation*}
Moreover, we have to treat also the pressure component of the momentum for each fluid network $j\in J$:
\begin{equation*}
\begin{split}
   \int_{\Omega}\alpha_j\nabla p_{jh}\cdot\boldsymbol{v}_h = & \int_{\Omega}\alpha_j\nabla\cdot(p_{jh} \mathrm{\mathbf{I}})\cdot\boldsymbol{v}_h = \\ = &
    - \sum_{K\in\partition}\int_{K}\alpha_j p_{jh}\mathrm{\mathbf{I}}:\nabla\boldsymbol{v}_h
    + \sum_{K\in\partition}\int_{\partial K}\alpha_j (p_{jh}\mathrm{\mathbf{I}}\cdot\boldsymbol{n})\cdot\boldsymbol{v}_h\mathrm{d}\sigma =  \\
    = &
    - \int_{\Omega}\alpha_j p_{jh}\mathrm{\mathbf{I}}:\nabla_h\boldsymbol{v}_h
    + \sum_{F\in\facesinternal\cup\facesDj}\int_{F}\alpha_j \averagel p_{jh}\mathrm{\mathbf{I}}\averager:\jjumpl\boldsymbol{v}_h\jjumpr\mathrm{d}\sigma + \int_{\Gamma_N}\alpha_j (p_{jh}\mathrm{\mathbf{I}}\cdot\boldsymbol{n})\cdot\boldsymbol{v}_h\mathrm{d}\sigma = \\ 
    = &
    - \int_{\Omega}\alpha_j p_{jh}(\nabla_h\cdot\boldsymbol{v}_h)
    + \sum_{F\in\facesinternal\cup\facesDj}\int_{F}\alpha_j \averagel p_{jh}\mathrm{\mathbf{I}}\averager:\jjumpl\boldsymbol{v}_h\jjumpr\mathrm{d}\sigma + \int_{\Gamma_N}\alpha_j (p_{jh}\mathrm{\mathbf{I}}\cdot\boldsymbol{n})\cdot\boldsymbol{v}_h\mathrm{d}\sigma  \qquad\forall \boldsymbol{v}_h\in\Vh \\ 
\end{split}
\end{equation*}
By exploiting the definitions of the bilinear forms, we arrive at the PolyDG semi-discrete formulation of the momentum equation for any $t\in(0,T]$:
\begin{equation}
\rho\left(\ddot{\boldsymbol{u}}_h(t),\boldsymbol{v}_h\right)_\Omega + \mathcal{A}_\mathrm{E}(\boldsymbol{u}_h(t),\boldsymbol{v}_h) - \sum_{k \in J}  \mathcal{B}_k(p_{kh}(t),\boldsymbol{v}_h) = F(\boldsymbol{v}_h)
     \qquad \forall\boldsymbol{v_h}\in \Vh.
\end{equation}
\par
Following similar arguments, we can rewrite the continuity equations for the fluid networks, from the Darcy flux component:
\begin{equation*}
    \begin{split}
    - \int_{\Omega} \nabla\cdot\left(\dfrac{\boldsymbol{\mathrm{K}}_j}{\mu_j}\nabla p_{jh}\right)q_{jh} = &
     - \sum_{K\in\partition} \int_{K} \nabla\cdot\left(\dfrac{\boldsymbol{\mathrm{K}}_j}{\mu_j}\nabla p_{jh}\right)q_{jh} = \\
    = & 
    \sum_{K\in\partition} \int_{K} \dfrac{\boldsymbol{\mathrm{K}}_j}{\mu_j}\nabla p_{jh}\cdot\nabla q_{jh} - \sum_{K\in\partition} \int_{\partial K} \left(\dfrac{\boldsymbol{\mathrm{K}}_j}{\mu_j}\nabla p_{jh}\cdot \boldsymbol{n}\right) q_{jh}= \\
    = & \int_{\Omega} \dfrac{\boldsymbol{\mathrm{K}}_j}{\mu_j}\nabla_h p_{jh}\cdot\nabla_h q_{jh} - \int_{\Gamma_N^j} h_j q_{jh} - \sum_{F\in\facesinternal\cup\facesDj} \int_{F} \dfrac{1}{\mu_j}\averagel\boldsymbol{\mathrm{K}}_j\nabla_h p_{jh}\averager\cdot\jumpl q_{jh}\jumpr  + \\
    & - \sum_{F\in\facesinternal\cup\facesDj} \int_{F} \dfrac{1}{\mu_j}\averagel\boldsymbol{\mathrm{K}}_j\nabla_h q_{jh}\averager\cdot\jumpl p_{jh}\jumpr  + \sum_{F\in\facesinternal\cup\facesDj} \int_{F} \zeta_j \jumpl p_{jh}\jumpr\cdot\jumpl q_{jh}\jumpr \qquad q_{jh}\in\Qh,
    \end{split}
\end{equation*}
\par
The continuity equation reads as for any $t\in(0,T]$:
\begin{equation}
c_j\left(\dot{p}_{jh}(t),q_{jh}\right)_{\Omega} + \mathcal{B}_j\left(q_{jh},
    \dot{\boldsymbol{u}}_h(t)\right) + 
    \mathcal{A}_{\mathrm{P}_j}(p_{jh}(t),q_{jh}) +  C_j\left((p_{kh})_{k\in J},q_{jh}\right)  = G_j(q_{jh}) \qquad \forall q_{jh}\in\Qh.   
\end{equation}
\par
To conclude the PolyDG semi-discrete formulation of the MPET problem reads as:
\par
\bigskip
Find $\boldsymbol{u}_h(t)\in \Vh$ and $p_{jh}(t) \in \Qh$ with $j\in J$ such that $\forall t>0$:
\begin{equation}
\begin{dcases}
     \rho\left(\ddot{\boldsymbol{u}}_h(t),\boldsymbol{v}_h\right)_\Omega + \mathcal{A}_\mathrm{E}(\boldsymbol{u}_h(t),\boldsymbol{v}_h) - \sum_{k \in J} \mathcal{B}_k(p_{kh}(t),\boldsymbol{v}_h) = F(\boldsymbol{v}_h),
     & \forall\boldsymbol{v_h}\in \Vh \\[8pt]
    c_j\left(\dot{p}_{jh}(t),q_{jh}\right)_{\Omega} + \mathcal{B}_j\left(q_{jh},
    \dot{\boldsymbol{u}}_h(t)\right) + 
    \mathcal{A}_{\mathrm{P}_j}(p_{jh}(t),q_{jh}) +  C_j\left((p_{kh})_{k\in J},q_{jh}\right)  = G_j(q_{jh}), & \forall q_{jh}\in\Qh  \\[8pt]
    \boldsymbol{u}_h(0)=\boldsymbol{u}_{0h}, & \mathrm{in}\;\Omega_h \\[8pt]
    \dot{\boldsymbol{u}}_h(0)=\boldsymbol{v}_{0h}, & \mathrm{in}\;\Omega_h \\[8pt]
    p_{jh}(0)=p_{j0h}, & \mathrm{in}\;\Omega_h\\[8pt]
    \boldsymbol{u}_h(t) = \boldsymbol{u}^\mathrm{D}_h(t), & \mathrm{on}\;\Gamma_D\\[8pt]
   q_{jh}(t)=q^\mathrm{D}_{jh}(t), & \mathrm{on}\;\Gamma_D^j
\end{dcases}
\end{equation}

\bibliographystyle{ieeetr}
\bibliography{sample.bib}

\end{document}